\documentclass[reqno,english]{amsart}
\usepackage[T1]{fontenc}
\usepackage[latin9]{inputenc}
\usepackage{color}
\usepackage{babel}
\usepackage{refstyle}
\usepackage{float}
\usepackage{mathrsfs}
\usepackage{mathtools}
\usepackage{enumitem}
\usepackage{amstext}
\usepackage{amsthm}
\usepackage{amssymb}
\usepackage{graphicx}
\usepackage[all]{xy}
\usepackage[unicode=true,pdfusetitle,
 bookmarks=true,bookmarksnumbered=false,bookmarksopen=false,
 breaklinks=false,pdfborder={0 0 0},pdfborderstyle={},backref=false,colorlinks=false]
 {hyperref}
\hypersetup{
 colorlinks=true,citecolor=blue,linkcolor=blue,linktocpage=true}

\makeatletter

\AtBeginDocument{\providecommand\secref[1]{\ref{sec:#1}}}
\AtBeginDocument{\providecommand\lemref[1]{\ref{lem:#1}}}
\AtBeginDocument{\providecommand\defref[1]{\ref{def:#1}}}
\AtBeginDocument{\providecommand\figref[1]{\ref{fig:#1}}}
\AtBeginDocument{\providecommand\propref[1]{\ref{prop:#1}}}
\AtBeginDocument{\providecommand\thmref[1]{\ref{thm:#1}}}
\AtBeginDocument{\providecommand\corref[1]{\ref{cor:#1}}}
\providecommand{\tabularnewline}{\\}
\RS@ifundefined{subsecref}
  {\newref{subsec}{name = \RSsectxt}}
  {}
\RS@ifundefined{thmref}
  {\def\RSthmtxt{theorem~}\newref{thm}{name = \RSthmtxt}}
  {}
\RS@ifundefined{lemref}
  {\def\RSlemtxt{lemma~}\newref{lem}{name = \RSlemtxt}}
  {}

\numberwithin{equation}{section}
\numberwithin{figure}{section}
\numberwithin{table}{section}
\theoremstyle{plain}
\newtheorem{thm}{\protect\theoremname}[section]
\theoremstyle{plain}
\newtheorem{lem}[thm]{\protect\lemmaname}
\theoremstyle{definition}
\newtheorem{defn}[thm]{\protect\definitionname}
\theoremstyle{plain}
\newtheorem{cor}[thm]{\protect\corollaryname}
\theoremstyle{remark}
\newtheorem*{rem*}{\protect\remarkname}
\theoremstyle{remark}
\newtheorem{rem}[thm]{\protect\remarkname}
\theoremstyle{plain}
\newtheorem{question}[thm]{\protect\questionname}
\theoremstyle{definition}
\newtheorem{example}[thm]{\protect\examplename}
\theoremstyle{plain}
\newtheorem{prop}[thm]{\protect\propositionname}
\theoremstyle{definition}
\newtheorem*{problem*}{\protect\problemname}

\providecommand{\MR}[1]{}

\usepackage{bbm}
\allowdisplaybreaks
\usepackage{needspace}

\newref{lem}{refcmd={Lemma \ref{#1}}}
\newref{thm}{refcmd={Theorem \ref{#1}}}
\newref{cor}{refcmd={Corollary \ref{#1}}}
\newref{sec}{refcmd={Section \ref{#1}}}
\newref{subsec}{refcmd={Section \ref{#1}}}
\newref{chap}{refcmd={Chapter \ref{#1}}}
\newref{prop}{refcmd={Proposition \ref{#1}}}
\newref{exa}{refcmd={Example \ref{#1}}}
\newref{tab}{refcmd={Table \ref{#1}}}
\newref{rem}{refcmd={Remark \ref{#1}}}
\newref{def}{refcmd={Definition \ref{#1}}}
\newref{fig}{refcmd={Figure \ref{#1}}}
\newref{claim}{refcmd={Claim \ref{#1}}}

\AtBeginDocument{
  
}

\makeatother

\providecommand{\corollaryname}{Corollary}
\providecommand{\definitionname}{Definition}
\providecommand{\examplename}{Example}
\providecommand{\lemmaname}{Lemma}
\providecommand{\problemname}{Problem}
\providecommand{\propositionname}{Proposition}
\providecommand{\questionname}{Question}
\providecommand{\remarkname}{Remark}
\providecommand{\theoremname}{Theorem}

\begin{document}
\title{Positive Definite Kernels, Algorithms, Frames, and Approximations}
\begin{abstract}
The main purpose of our paper is a new approach to design of algorithms
of Kaczmarz type in the framework of operators in Hilbert space. Our
applications include a diverse list of optimization problems, new
Karhunen-Lo\`eve transforms, and Principal Component Analysis (PCA)
for digital images. A key feature of our algorithms is our use of
recursive systems of projection operators. Specifically, we apply
our recursive projection algorithms for new computations of PCA probabilities
and of variance data. For this we also make use of specific reproducing
kernel Hilbert spaces, factorization for kernels, and finite-dimensional
approximations. Our projection algorithms are designed with view to
maximum likelihood solutions, minimization of \textquotedblleft cost\textquotedblright{}
problems, identification of principal components, and data-dimension
reduction.
\end{abstract}

\author{Palle E.T. Jorgensen}
\address{(Palle E.T. Jorgensen) Department of Mathematics, The University of
Iowa, Iowa City, IA 52242-1419, U.S.A.}
\email{palle-jorgensen@uiowa.edu}
\author{Myung-Sin Song}
\address{(Myung-Sin Song) Department of Mathematics and Statistics, Southern
Illinois University Edwardsville, Edwardsville, IL 62026, USA}
\email{msong@siue.edu}
\author{James Tian}
\address{(James F. Tian) Mathematical Reviews, 416 4th Street Ann Arbor, MI
48103-4816, U.S.A.}
\email{jft@ams.org}

\maketitle
\tableofcontents{}

\section{Introduction}

While positive-definite (p.d.) kernels date back to the 1950ties,
recently they have found striking and new applications in applied
mathematics. Indeed, there are by now multiple and diverse applications,
most of quite recent vintage. In outline, a p.d. kernel is a generalization
of the notion of positive-definite function, or p.d. matrix. Initially,
p.d. kernels were introduced with view to the use for the solution
of integral operator equations arising in PDE theory (especially boundary
value problems), potential theory, and in optimization problems. Since
their inception, positive-definite kernels and their associated Hilbert
spaces, reproducing kernel Hilbert spaces (RKHSs) have come up in
entirely new areas of mathematics. A RKHS is a Hilbert space $\mathscr{H}$
of functions on a set, say $X$, with the property that for $f$ in
$\mathscr{H}$, $f(x)$ is continuous in the norm of $\mathscr{H}$.
RKHSs arise naturally in Fourier analysis, probability theory, stochastic
processes, operator theory, complex function-theory, interpolation,
signal/image processing, sampling, moment problems, machine learning,
embedding problem from geometry and analysis, and information theory,
among other areas. Readers are referred to \cite{MR3796644,MR3796634,MR3711878,MR1898684,HERR2018,MR3800275,MR3439812,MR2500924,MR3896982}.

In details, every p.d. kernel $K$ on a set $X$ has  an associated
RKHS $\mathscr{H}_{K}$. The reproducing axiom is as follows: For
all $x\in X$, the function $K\left(\cdot,x\right)$ on $X$ is in
$\mathscr{H}_{K}$, and 
\begin{equation}
F\left(x\right)=\left\langle K\left(\cdot,x\right),F\right\rangle _{\mathscr{H}_{K}}\label{eq:a6}
\end{equation}
holds for all $x\in X$ and all $\ensuremath{F\in\mathscr{H}_{K}}$.
Eq (\ref{eq:a6}) is referred to as \emph{the reproducing property}. 

A kernel $X\times X\xrightarrow{\;K\;}\mathbb{C}$ is called \emph{positive
definite} (p.d.) if $\forall n\in\mathbb{N}$, $\forall\left\{ \alpha_{i}\right\} _{1}^{n}$,
$\forall\left\{ x_{i}\right\} _{1}^{n}$, $\alpha_{i}\in\mathbb{C}$,
$x_{i}\in X$, we have 
\begin{equation}
\sum_{i}\sum_{j}\overline{\alpha}_{i}\alpha_{j}K\left(x_{i},x_{j}\right)\geq0.\label{eq:I3}
\end{equation}
Given $K$ p.d., a realization of the RKHS $\mathscr{H}_{K}$ is to
take the completion of the linear span of $\left\{ K\left(\cdot,x\right)\right\} _{x\in X}$
with respect to the norm
\begin{equation}
\left\Vert \sum\alpha_{i}K\left(\cdot,x_{i}\right)\right\Vert _{\mathscr{H}_{K}}^{2}\coloneqq\sum_{i}\sum_{j}\overline{\alpha}_{i}\alpha_{j}K\left(x_{i},x_{j}\right).
\end{equation}

\textbf{Organization.} The paper is organized as follows: in sections
\ref{sec:pdk} and \ref{sec:ofp}, we present systems $\left(K,\mu\right)$
in duality, where $K$ is a fixed p.d. kernel, and choices of $\mu$
runs through an associated family of measures. This duality will play
an important role, for example in a unified approach to (i) a class
of optimization problems, and to (ii) realization of p.d. kernels
as integral operators. In particular, we show how properties of $\mu$
reflect themselves in spectral theory for the corresponding integral
operator. \secref{rp} offers a framework for sampling. The starting
point is a p.d. kernel defined on $X\times X$ for a given set $X$.
The questions we address are: What countably discrete subsets $V$
of $X$ provide sets of sample points for suitable classes of functions
on $X$, and algorithms of Kaczmarz type, sect. \ref{sec:kac}. The
answer to this, and related, questions depends on configuration of
systems of finite subsets of $X$, addressed in section \ref{sec:fp};
with examples in sections \ref{sec:PCA}and \ref{sec:pw}. In section
\ref{sec:gp}, we then present a new harmonic analysis of general
Gaussian processes, with the use of our RKHS-analysis, and projection
based algorithms, from inside the paper.

Our projection approach to algorithms is motivated in part by principal
component analysis (PCA), and its applications in dimension reduction
and manifold learning, which in turn, extends earlier work on Monte
Carlo, and Karhunen--Lo\`eve analysis, also known as the Kosambi--Karhunen--Lo\`eve
approach. See, e.g., \cite{MR2459323,MR2362796,MR3780557,MR3922239,jorgensen2019dimension}.

\textbf{Background material. }Our paper is interdisciplinary, and
we have aimed for multiple target audiences. Inside our paper, we
have therefore added explanations for the benefit of readers from
neighboring areas. Specifically, background material is included for
such notions as Karhunen-Lo\`eve transforms, Principal Component
Analysis (PCA), Monte-Carlo simulations, and Paley-Wiener spaces.
Monte-Carlo refers to the use randomness in order to solve problems
that might be deterministic in principle. In rough outline, Karhunen-Lo\`eve
transforms, and PCAs, are algorithmic tools, used data analysis. Our
present emphasis is that of digital image algorithms. Such algorithms
serve to produce the best possible bases for image expansions. This
area is combined with Monte-Carlo simulations, i.e., random simulations
that make use of computer-generated sequences of independent, centered
real stochastic process (typically, Gaussian $N(0,1)$); as well as
related Paley-Wiener spaces. A key use of the Karhunen--Lo\`eve
theorem is algorithmic designs of canonical and orthogonal representations
of images and signals. Paley-Wiener spaces offer a useful framework
for such reconstruction algorithms.

\section{\label{sec:pdk}Positive definite kernels and measures}

The notion of \emph{positive definite} (p.d.) functions, usually called
kernels, adapts well to many diverse optimization problems, arising
for example in \emph{machine leaning}, where training data are chosen
to best accommodate features which serve as input in the \textquotedblleft learning.\textquotedblright{}
The notion of \emph{feature spaces} and \emph{feature maps} are defined
from specific Hilbert spaces (which we will call the feature space);
details below. For most application a suitable infinite-dimensional
framework is essential. So in the general case, if a p.d. kernel function
$K$ is given on a $X\times X$, where $X$ is a set, then there is
a rich variety of \emph{factorizations}. More precisely, this means
that there are many choices of feature maps from $X$ into some Hilbert
space (a choice of \emph{feature space}) which yield back $K$ in
a factorization. Among the feature spaces, the RKHS $\mathscr{H}_{K}$
associated with $K$ is universal (minimal) in the sense that $\mathscr{H}_{K}$
is isometrically \textquotedblleft included\textquotedblright{} in
any of the feature spaces which arise in factorizations for $K$.
For many optimization problems it will be helpful to adapt factorizations
in such a way that the feature Hilbert space is an $L^{2}$ space,
and this will be the present focus.

Motivated by related applications to the study of stochastic processes,
it is of special significance to focus on the cases when the family
of feature spaces may be chosen in the form $L^{2}\left(\mu\right)$.
This raises the question of which measures $\mu$ are right for a
particular kernel $K$ and its associated RKHS. The answer to this
depends on the particular application at hand. 

We begin with the correspondence between reproducing kernel Hilbert
spaces (RKHSs) and integral operators and their spectral resolutions. 

To make precise the notion of Borel measures, we must first introduce
a suitable sigma-algebra of Borel subsets of $X$. The idea is as
follows: Starting with a p.d. kernel $K$ on $X\times X$, we note
below that there is then an induced metric $d_{K}$ on $X$ which
makes $K$ jointly continuous as a function on $X\times X$. But of
course, $d_{K}$ also introduces a (metric) topology on $X$, and
therefore the associated sigma-algebra $\mathscr{B}_{K}$ generated
by the $d_{K}$ neighborhoods; see details below. For some purposes,
it may be convenient to pass to the metric completion of $\left(X,d_{K}\right)$.

Let $X\times X\xrightarrow{\;K\;}\mathbb{R}$ be a positive definite
(p.d.) kernel on a set $X$. Let $\mathscr{H}_{K}$ be the associated
RKHS with norm $\left\Vert \cdot\right\Vert _{\mathscr{H}_{K}}$.
Let $\mathcal{M}\left(X\right)$ be the set of all Borel measures
on $X$.
\begin{lem}
Given $X\times X\xrightarrow{\;K\;}\mathbb{R}$ as above, the set
$X$ can be equipped with a metric 
\begin{equation}
d_{K}\left(x,y\right)\coloneqq\sqrt{K\left(x,x\right)+K\left(y,y\right)-2K\left(x,y\right)},\;\forall x,y\in X.\label{eq:A3}
\end{equation}
\end{lem}

\begin{proof}
One checks that $d_{K}\left(x,y\right)=\left\Vert K_{x}-K_{y}\right\Vert _{\mathscr{H}_{K}}$,
where $K_{s}=K\left(\cdot,s\right)$ for all $s\in X$. 
\end{proof}
\begin{lem}
The kernel $K$ is continuous on $X\times X$ with respect to the
product topology induced by $d_{K}\times d_{K}$. 
\end{lem}

\begin{proof}
Indeed, for all $a,b,c,d\in X$, it holds that
\begin{align*}
\left|K\left(a,b\right)-K\left(c,d\right)\right| & =\left|\left\langle K_{a},K_{b}\right\rangle _{\mathscr{H}_{K}}-\left\langle K_{c},K_{d}\right\rangle _{\mathscr{H}_{K}}\right|\\
 & \leq\left|\left\langle K_{a},K_{b}-K_{d}\right\rangle _{\mathscr{H}_{K}}\right|+\left|\left\langle K_{a}-K_{c},K_{d}\right\rangle _{\mathscr{H}_{K}}\right|\\
 & \leq\left\Vert K_{a}\right\Vert _{\mathscr{H}_{K}}\left\Vert K_{b}-K_{d}\right\Vert _{\mathscr{H}_{K}}+\left\Vert K_{d}\right\Vert _{\mathscr{H}_{K}}\left\Vert K_{a}-K_{c}\right\Vert _{\mathscr{H}_{K}},
\end{align*}
so the assertion follows. 
\end{proof}
\begin{defn}
\label{def:reg}Let $\mu\in\mathcal{M}\left(X\right)$, i.e., a Borel
measure on $\left(X,\mathscr{F}\right)$. The measure $\mu$ is said
to be in $Reg\left(K\right)$ if 
\begin{equation}
K_{x}\coloneqq K\left(\cdot,x\right)\in L^{2}\left(X,\mathscr{F},\mu\right),\;\forall x\in X.\label{eq:A2}
\end{equation}
\end{defn}

If (\ref{eq:A2}) is assumed to hold for some $\sigma$-finite measure
$\mu$ on $\left(X,\mathscr{F}\right)$ and $\mathscr{F}=\mathscr{B}_{K}$,
the Borel $\sigma$-algebra defined from the open sets with respect
to the metric $d_{K}$, then we get a well defined linear operator
$T_{\mu}:\mathscr{H}_{K}\rightarrow L^{2}\left(\mu\right)$, by 
\begin{equation}
T_{\mu}\left(\sum\nolimits _{j}c_{j}K\left(\cdot,x_{j}\right)\right)\coloneqq\sum\nolimits _{j}c_{j}K\left(\cdot,x_{j}\right)\label{eq:A4}
\end{equation}
where $c_{j}\in\mathbb{R}$, $x_{j}\in X$ with a finite sum. However,
the operator $T_{\mu}$ may be unbounded. It is always closable (\lemref{dense}).
We now turn to the adjoint operator $T_{\mu}^{*}:L^{2}\left(\mu\right)\rightarrow\mathscr{H}_{K}$. 
\begin{lem}
\label{lem:Tadj}Fix $K$, and $\mu\in Reg\left(K\right)$ (see \defref{reg}).
Let $T_{\mu}$ be as in (\ref{eq:A4}), then 
\begin{equation}
\left(T_{\mu}^{*}\varphi\right)\left(x\right)=\int_{X}K\left(y,x\right)\varphi\left(y\right)d\mu\left(y\right),\;\forall\varphi\in dom\left(T_{\mu}^{*}\right).\label{eq:A5}
\end{equation}
\end{lem}

\begin{proof}
For all $x\in X$, and all $\varphi\in dom\left(T_{\mu}^{*}\right)$,
\[
\left\langle T_{\mu}^{*}\varphi,K\left(\cdot,x\right)\right\rangle _{\mathscr{H}_{K}}=\big\langle\varphi,\underset{=K\left(\cdot,x\right)}{\underbrace{T_{\mu}\left(K\left(\cdot,x\right)\right)}}\big\rangle_{L^{2}\left(\mu\right)}=\int_{X}\varphi\left(y\right)K\left(y,x\right)\mu\left(dy\right).
\]
\end{proof}
\begin{lem}
\label{lem:dense}Let $T_{\mu}^{*}$ be as above. Then $\varphi\in dom\left(T_{\mu}^{*}\right)\subset L^{2}\left(\mu\right)$
if and only if 
\begin{gather}
\int_{X}\int_{X}\varphi\left(y\right)\varphi\left(z\right)K\left(y,z\right)\mu\left(dy\right)\mu\left(dz\right)<\infty.\label{eq:A9}
\end{gather}

Moreover, $dom\left(T_{\mu}^{*}\right)$ is dense in $L^{2}\left(\mu\right)$,
and so $T$ is closable. 
\end{lem}

\begin{proof}
Let $\varphi\in L^{2}\left(\mu\right)$. If (\ref{eq:A9}) holds,
then $\int_{X}K\left(y,\cdot\right)\varphi\left(y\right)\mu\left(dy\right)$
belongs to $\mathscr{H}_{K}$, and so $T_{\mu}^{*}\varphi$ is well-defined,
i.e., $\varphi\in dom\left(T_{\mu}^{*}\right)$. Conversely, if $\varphi\in dom\left(T_{\mu}^{*}\right)$,
then 
\begin{align*}
\left\Vert T_{\mu}^{*}\varphi\right\Vert _{\mathscr{H}_{K}}^{2} & =\left\langle \int_{X}K\left(y,\cdot\right)\varphi\left(y\right)\mu\left(dy\right),\int_{X}K\left(z,\cdot\right)\varphi\left(z\right)\mu\left(dz\right)\right\rangle _{\mathscr{H}_{K}}\\
 & =\int_{X}\int_{X}\varphi\left(y\right)\varphi\left(z\right)K\left(y,z\right)\mu\left(dy\right)\mu\left(dz\right)<\infty
\end{align*}
which is (\ref{eq:A9}).
\end{proof}
\begin{cor}
Let $\mu\in Reg\left(K\right)$ as above. For all $x,z\in X$, then
\begin{align}
\left(T_{\mu}^{*}T_{\mu}K\left(\cdot,x\right)\right)\left(z\right) & =\int_{X}K\left(z,y\right)K\left(y,x\right)\mu\left(dy\right)\label{eq:A6}\\
 & =\left\langle K\left(\cdot,z\right),K\left(\cdot,x\right)\right\rangle _{L^{2}\left(\mu\right)}.\nonumber 
\end{align}
Moreover, 
\begin{equation}
T_{\mu}T_{\mu}^{*}\varphi=\int\varphi\left(y\right)K\left(y,\cdot\right)\mu\left(dy\right)\label{eq:A8}
\end{equation}
with $\varphi\in dom\left(T_{\mu}^{*}\right)$, the domain of the
operator $T_{\mu}^{*}$; i.e., $\varphi\in dom\left(T_{\mu}^{*}\right)$
if the RHS of (\ref{eq:A8}) is in $\mathscr{H}_{K}$. 
\end{cor}

\begin{cor}
Let $K$, $\mathscr{H}_{K}$, and $\mu$ be as in \defref{reg}, and
let $T_{\mu}$ be the associated operator, see \lemref{Tadj}. Then
for $\varphi\in dom(T_{\mu}^{*})\subset L^{2}\left(\mu\right)$, and
$\alpha\in\mathbb{R}_{+}$, consider the following optimization problem:
\begin{equation}
\mathscr{H}_{K}\ni f=\mathop{argmin}\left\{ \left\Vert \varphi-T_{\mu}f\right\Vert _{L^{2}\left(\mu\right)}^{2}+\alpha\left\Vert f\right\Vert _{\mathscr{H}_{K}}^{2}\right\} .\label{eq:ko}
\end{equation}
The solution is 
\[
f=\left(\alpha I+T_{\mu}^{*}T_{\mu}\right)^{-1}\underset{\in\mathscr{H}_{K}}{\underbrace{\left(T_{\mu}^{*}\varphi\right)}}.
\]
\end{cor}

\begin{proof}
This is a direct computation; see also \cite{MR3888850,jorgensen2019dimension,MR3642406}.

Sketch. Define $W:\mathscr{H}_{K}\rightarrow\mathscr{H}_{K}\times L^{2}\left(\mu\right)$,
$Wf=\left(f,T_{\mu}f\right)$, where 
\[
\left\Vert \left(f,g\right)\right\Vert _{\mathscr{H}_{K}\times L^{2}\left(\mu\right)}^{2}\coloneqq\alpha\left\Vert f\right\Vert _{\mathscr{H}_{K}}^{2}+\left\Vert g\right\Vert _{L^{2}\left(\mu\right)}^{2}.
\]
Then, $W^{*}\left(f,g\right)=\alpha f+T_{\mu}^{*}g$, and by the standard
least square approximation, 
\begin{eqnarray*}
\text{RHS}_{\left(\ref{eq:ko}\right)} & = & \mathop{argmin}\left\{ \left\Vert Wf-\left(0,\varphi\right)\right\Vert _{\mathscr{H}_{K}\times L^{2}\left(\mu\right)}^{2}\right\} \\
 & = & \left(W^{*}W\right)^{-1}W^{*}\left(\left(0,\varphi\right)\right)\\
 & = & \left(\alpha I+T_{\mu}^{*}T_{\mu}\right)^{-1}T_{\mu}^{*}\varphi.
\end{eqnarray*}
\end{proof}
\begin{rem*}
Recall that the functions 
\begin{equation}
\left(x_{1},x_{2},\cdots,x_{n}\right)\longmapsto\left(K\left(\cdot,x_{1}\right),K\left(\cdot,x_{2}\right),\cdots,K\left(\cdot,x_{n}\right)\right)\label{eq:A7}
\end{equation}
are often called feature functions, and $\mathscr{H}_{K}$ is interpreted
as a feature space.

\textbf{Kernel-optimization. }One of the more recent applications
of RKHSs is the kernel-optimization. It refers to training-data and
feature spaces in the context of machine learning. In numerical analysis,
a popular version of the method is used to produce splines from sample
points; and to create best spline-fits. In statistics, there are analogous
optimization problems going by the names \textquotedblleft least-square
fitting,\textquotedblright{} and \textquotedblleft maximum-likelihood\textquotedblright{}
estimation. 

What these methods have in common is a minimization (or a max problem)
involving a \textquotedblleft quadratic\textquotedblright{} expression
$Q$ (see (\ref{eq:ko})) with two terms: (i) a $L^{2}$-square applied
to a difference, and (ii) a penalty term which is a RKHS norm-squared.
In the application to determination of splines, the penalty term may
be a suitable Sobolev normed-square; i.e., $L^{2}$ norm-squared applied
to a chosen number of derivatives. Hence non-differentiable choices
will be \textquotedblleft penalized.\textquotedblright{} 
\end{rem*}
\begin{defn}
Let $\mathscr{H}$ be a Hilbert space with inner product denoted $\left\langle \cdot,\cdot\right\rangle $,
or $\left\langle \cdot,\cdot\right\rangle _{\mathscr{H}}$ when there
is more than one possibility to consider. Let $J$ be a countable
index set, and let $\left\{ w_{j}\right\} _{j\in J}$ be an indexed
family of non-zero vectors in $\mathscr{H}$. We say that $\left\{ w_{j}\right\} _{j\in J}$
is a \emph{frame }for $\mathscr{H}$ iff (Def.) there are two finite
positive constants $A$ and $B$ such that
\begin{equation}
A\left\Vert u\right\Vert _{\mathscr{H}}^{2}\leq\sum_{j\in J}\left|\left\langle w_{j},u\right\rangle _{\mathscr{H}}\right|^{2}\leq B\left\Vert u\right\Vert _{\mathscr{H}}^{2}\label{eq:en1}
\end{equation}
holds for all $u\in\mathscr{H}$. We say that it is a \emph{Parseval}
frame if $A=B=1$. 

For references to the theory and application of \emph{frames}, see
e.g., \cite{MR3009685,MR2538596,MR3121682}.
\end{defn}

\begin{lem}
If $\left\{ w_{j}\right\} _{j\in J}$ is a Parseval frame in $\mathscr{H}$,
then the (analysis) operator $T:\mathscr{H}\longrightarrow l^{2}\left(J\right)$,
\begin{equation}
Tu=\left(\left\langle w_{j},u\right\rangle _{\mathscr{H}}\right)_{j\in J}\label{eq:en2}
\end{equation}
is well-defined and isometric. Its adjoint $T^{*}:l^{2}\left(J\right)\longrightarrow\mathscr{H}$
is given by 
\begin{equation}
T^{*}\left(\left(\gamma_{j}\right)_{j\in J}\right):=\sum_{j\in J}\gamma_{j}w_{j}\label{eq:en3}
\end{equation}
and the following hold:

\begin{enumerate}
\item The sum on the RHS in (\ref{eq:en3}) is norm-convergent;
\item $T^{*}:l^{2}\left(J\right)\longrightarrow\mathscr{H}$ is co-isometric;
and for all $u\in\mathscr{H}$, we have 
\begin{equation}
u=T^{*}Tu=\sum_{j\in J}\left\langle w_{j},u\right\rangle w_{j}\label{eq:en4}
\end{equation}
where the RHS in (\ref{eq:en4}) is norm-convergent. 
\end{enumerate}
\end{lem}

\begin{proof}
The details are standard in the theory of frames; see the cited papers
above. Note that (\ref{eq:en1}) for $A=B=1$ simply states that $V$
in (\ref{eq:en2}) is isometric, and so $T^{*}T=I_{\mathscr{H}}=$
the identity operator in $\mathscr{H}$, and $TT^{*}=$ the projection
onto the range of $V$.
\end{proof}
\begin{rem}
We may always get an orthonormal basis (ONB) or a Parseval frame $\left\{ f_{i}\right\} _{i\in J}$
in $\mathscr{H}_{K}$, where the index $J$ is usually $\mathbb{N}_{0}=\left\{ 0,1,2,3,\cdots\right\} $. 
\end{rem}

\begin{lem}
Given a Parseval frame $\left\{ f_{i}\right\} _{i\in\mathbb{N}_{0}}$
in $\mathscr{H}_{K}$, then
\begin{equation}
K\left(x,y\right)=\sum_{i\in\mathbb{N}_{0}}f_{i}\left(x\right)f_{i}\left(y\right),\quad x,y\in X.\label{eq:A11}
\end{equation}
If $\mu\in Reg\left(K\right)$, then 
\begin{equation}
\varphi\in L^{2}\left(\mu\right)\Longleftrightarrow\left(\left\langle \varphi,f_{i}\right\rangle _{L^{2}\left(\mu\right)}\right)_{i}\in l^{2}\left(\mathbb{N}_{0}\right)\label{eq:A12}
\end{equation}
\end{lem}

\begin{proof}
Given $\left\{ f_{i}\right\} _{i\in\mathbb{N}_{0}}$ in $\mathscr{H}_{K}$,
one checks that 
\[
K\left(\cdot,y\right)=K_{y}\left(\cdot\right)=\sum_{i\in\mathbb{N}_{0}}\left\langle f_{i},K_{y}\right\rangle f_{i}\left(\cdot\right)=\sum_{i\in\mathbb{N}_{0}}f_{i}\left(y\right)f_{i}\left(\cdot\right)
\]
which is (\ref{eq:A11}).
\end{proof}
\begin{question}
Find the spectrum of $T_{\mu}$ as $\mu$ varies in $Reg\left(K\right)$.
Actually it is the selfadjoint operator $T_{\mu}T_{\mu}^{*}:L^{2}\left(\mu\right)\rightarrow L^{2}\left(\mu\right)$,
whose spectrum we want. 
\end{question}

It is important to keep in mind that we get a family of operators
$T_{\mu}$ indexed by $\mu\in Reg\left(K\right)$. They are densely
defined operators in the Hilbert space $\mathscr{H}_{K}$, but the
spectra may vary with $\mu$. 

Fix a p.d. kernel $X\times X\xrightarrow{\;K\;}\mathbb{R}$ on a set
$X$. Assign the Borel $\sigma$-algebra $\mathscr{B}_{K}$ and measure
$\mu$, and get $\left(X,\mathscr{B}_{K}\right)$ as above. What is
the interconnection between the following three conditions?
\begin{enumerate}
\item $K\left(\cdot,x\right)\in L^{2}\left(\mu\right)$, for all $x\in X$;
\item $\left\{ f_{i}\right\} \subset L^{2}\left(\mu\right)$, for all ONB
$\left\{ f_{i}\right\} $ in $\mathscr{H}_{K}$;
\item How is the system $\left\{ f_{i}\right\} $ related to the spectral
properties of $T_{\mu}T_{\mu}^{*}:L^{2}\left(\mu\right)\rightarrow L^{2}\left(\mu\right)$?
\end{enumerate}
A related question: Consider a p.d. kernel $K$ on $X\times X$. Suppose
$X$ is discrete and countable, equipped with the counting measure
$\gamma$ on $X$, i.e., 
\begin{equation}
\gamma\left(B\right)=\#\left(B\right),\;\forall B\subset X.\label{eq:A13}
\end{equation}
When is 
\begin{equation}
\gamma\in Reg\left(K\right)?\label{eq:A14}
\end{equation}
We shall give a precise solution to (\ref{eq:A14}) in the following
section.

\section{\label{sec:ofp}Operators from p.d. kernels, and some of their applications}

Here, we consider a special case in the setting of \secref{pdk},
where a p.d. kernel admits a factorization in some $L^{2}\left(\mu\right)$-space
with $\mu\in Reg\left(K\right)$. 

Let $X\times X\xrightarrow{\;K\;}\mathbb{R}$ be a p.d. kernel on
a set $X$. Let $\mathscr{H}_{K}$ be the associated RKHS. Assume:
\begin{enumerate}
\item $\mu\in Reg\left(K\right)$;
\item $\left\{ K_{x}\right\} _{x\in X}$ is dense in $L^{2}\left(\mu\right)$;
\item for all $F\in\mathscr{H}_{K}$, 
\[
\left\Vert F\right\Vert _{\mathscr{H}_{K}}\geq\left\Vert F\right\Vert _{L^{2}\left(\mu\right)}.
\]
\end{enumerate}
\begin{lem}
\label{lem:db1}Let $(\mathscr{H}_{j},\left\Vert \cdot\right\Vert _{j})$,
$j=1,2$, be a pair of Hilbert spaces. Suppose the inclusion map $J:\mathscr{H}_{1}\hookrightarrow\mathscr{H}_{2}$
has dense image, and $\left\Vert x\right\Vert _{1}\geq\left\Vert x\right\Vert _{2}$,
for all $x\in\mathscr{H}_{1}$. Then there exists a unique positive
(selfadjoint) operator $A\geq1$, such that $\mathscr{H}_{1}=dom\left(A^{1/2}\right)$
and 
\[
\left\langle x,y\right\rangle _{1}=\left\langle A^{1/2}x,A^{1/2}y\right\rangle _{2}
\]
for all $x,y\in\mathscr{H}_{1}$. 
\end{lem}

\begin{proof}
It can be verified that 
\begin{equation}
A\coloneqq\left(JJ^{*}\right)^{-1}\label{eq:A1}
\end{equation}
is the unique positive operator having the desired properties. For
more details, see e.g., \cite{MR0282379,MR1009163,MR1255973,MR0493420}
and \cite{MR3390972}.
\end{proof}
\begin{cor}
Let $\mathscr{H}_{1}=\mathscr{H}_{K}$ and $\mathscr{H}_{2}=L^{2}\left(\mu\right)$.
In view of \lemref{db1}, then $\Phi:X\rightarrow L^{2}\left(\mu\right)$,
defined as 
\[
\Phi\left(x\right)=A^{1/2}K_{x},\quad x\in X,
\]
is a feature map for the p.d. kernel $K$, and $L^{2}\left(\mu\right)$
is the corresponding feature space. 
\end{cor}

\begin{cor}
Let $A$ be as in (\ref{eq:A1}). Suppose $A^{-1}$ is compact with
spectral decomposition 
\[
A^{-1}=\sum_{i=1}^{\infty}\lambda_{i}\left|u_{i}\left\rangle \right\langle u_{i}\right|,
\]
where $1\geq\lambda_{i}\geq\lambda_{i+1}>0$, $\lambda_{i}\rightarrow0$
as $i\rightarrow\infty$, and $\left\{ u_{i}\right\} $ is an ONB
in $L^{2}\left(\mu\right)$.

Then, for all $x,y\in X$, it holds that 
\[
K\left(x,y\right)=\sum_{i=1}^{\infty}\lambda_{i}^{-1}\left\langle K_{x},u_{i}\right\rangle _{L^{2}}\left\langle u_{i},K_{y}\right\rangle _{L^{2}}.
\]
\end{cor}

\begin{proof}
One checks that
\begin{align*}
K\left(x,y\right) & =\left\langle K_{x},K_{y}\right\rangle _{\mathscr{H}_{K}}\\
 & =\left\langle \Phi\left(x\right),\Phi\left(y\right)\right\rangle _{L^{2}}\\
 & =\left\langle A^{1/2}K_{x},A^{1/2}K_{y}\right\rangle _{L^{2}}\\
 & =\sum_{i=1}^{\infty}\left\langle A^{1/2}K_{x},u_{i}\right\rangle _{L^{2}}\left\langle u_{i},A^{1/2}K_{y}\right\rangle _{L^{2}}\\
 & =\sum_{i=1}^{\infty}\left\langle K_{x},A^{1/2}u_{i}\right\rangle _{L^{2}}\left\langle A^{1/2}u_{i},K_{y}\right\rangle _{L^{2}}\\
 & =\sum_{i=1}^{\infty}\lambda_{i}^{-1}\left\langle K_{x},u_{i}\right\rangle _{L^{2}}\left\langle u_{i},K_{y}\right\rangle _{L^{2}}.
\end{align*}
\end{proof}

\section{\label{sec:rp}Restrictions of Positive Definite Kernels}

Let $K:X\times X\rightarrow\mathbb{C}$ be a positive definite kernel
defined on $X\times X$ (where $X$ is a fixed set); i.e., it is assumed
that for all finite subset $F\subset X$, and scalars $\left(c\left(x\right)\right)_{x\in F}$
in $\mathbb{C}^{\left|F\right|}$ we have 
\begin{equation}
\underset{F\times F}{\sum\sum}\overline{c\left(x\right)}c\left(y\right)K\left(x,y\right)\geq0.\label{eq:rp1}
\end{equation}
As before the RKHS will be denoted $\mathscr{H}_{K}$. 

If $V\subset X$ is a countably infinite subset, then 
\begin{equation}
K_{V}\coloneqq K\left(\cdot,\cdot\right)\big|_{V\times V}\label{eq:rp2}
\end{equation}
may be considered as an $\infty\times\infty$ matrix with $V$ serving
as row $\times$ column index. 

Consider the standard $l^{2}$-space $l^{2}\left(V\right)$ with dense
subspace $l_{0}^{2}\left(V\right)$ consisting of finitely supported
sequences, i.e., 
\begin{equation}
c\in l_{0}^{2}\left(V\right)\underset{\text{def}}{\Longleftrightarrow}\text{\ensuremath{\exists F} finite s.t. \ensuremath{c\equiv0} in \ensuremath{V\backslash F}.}\label{eq:rp3}
\end{equation}

We now turn to the discretized version of the condition we introduced
in \defref{reg} in the general context of sigma-finite measures.
Since we are now making selection of countably discrete subsets $V$
of $X$, we will have a condition for each choice of $V$. See details
below: 
\begin{rem}
When $K$, $X$, and $V\subset X$ are specified as in (\ref{eq:rp2}),
we shall impose the following $l^{2}\left(V\right)$ condition, which
restricts the possible choice of countable discrete subsets $V\subset X$: 

We shall assume that for all $y\in V$, 
\begin{equation}
\sum_{x\in V}\left|K\left(x,y\right)\right|^{2}\leq C_{y}^{\left(V\right)}<\infty,\label{eq:rp3a}
\end{equation}
with the constant $C_{y}^{\left(V\right)}$ depending on $y$, and
choice of $V$. 
\end{rem}

\begin{example}
If $X=\mathbb{R}$, and $K=$ the Shannon since kernel, i.e., 
\[
K\left(x,y\right)=\frac{\sin\pi\left(x-y\right)}{\pi\left(x-y\right)},\quad\forall\left(x,y\right)\in\mathbb{R}\times\mathbb{R}
\]
then one checks that condition (\ref{eq:rp3a}) is satisfied when
$V=\left(x_{n}\right)_{n\in\mathbb{N}}\subset\mathbb{R}\backslash\left\{ 0\right\} $
satisfying 
\begin{equation}
\sum_{n\in\mathbb{N}}\left|x_{n}\right|^{-2}<\infty.\label{eq:rp3b}
\end{equation}
The verification is a direct computation; it is illustrated by the
graph below (\figref{rp1}). 
\end{example}

\begin{figure}[H]
\includegraphics[width=0.4\columnwidth]{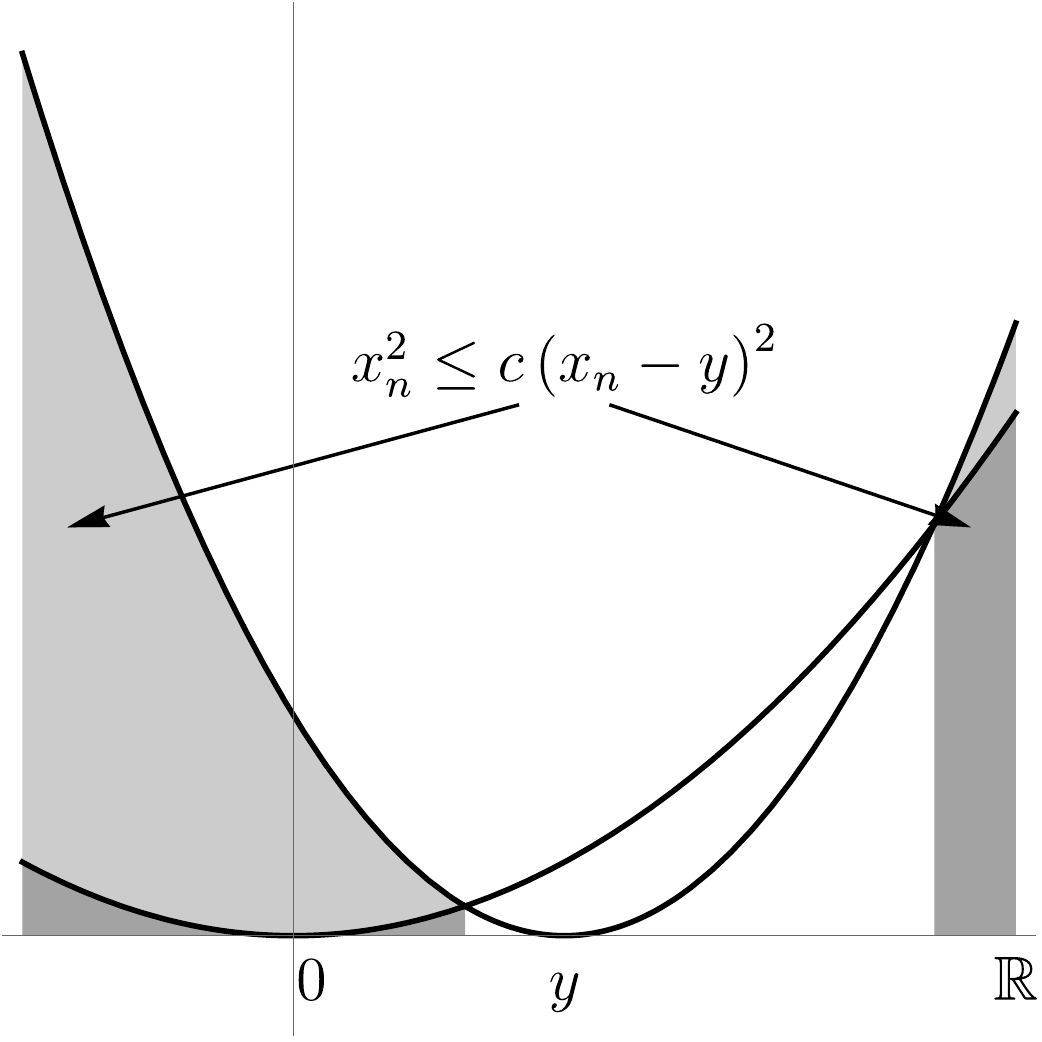}

\caption{\label{fig:rp1}Illustration of condition (\ref{eq:rp3a}) for the
Shannon sinc kernel. Assume $c>1$, then the inequality in (\ref{eq:rp3a})
is satisfied when $\left|x_{n}\right|$ is large. In the shaded region,
$x_{n}^{2}\protect\leq c\left(x_{n}-y\right)^{2}$. }
\end{figure}

We now return to the general case. Note that the action on $l^{2}\left(V\right)$
given by 
\begin{equation}
\left(K_{V}c\right)\left(x\right)=\sum_{y\in F}K_{V}\left(x,y\right)c\left(y\right)\label{eq:rp4}
\end{equation}
is well defined. Note the finite set $F$ from (\ref{eq:rp3}) depends
on $c$. 

Hence, when $V$ is given, assumed countably infinite, it is of interest
to consider the case when $K_{V}$, as in (\ref{eq:rp4}), defines
a bounded operator in $l^{2}\left(V\right)$; and when this operator
is invertible. 
\begin{prop}
\label{prop:rp1}Let $K$, $X$, and $V$ be as above, then the operator
$K_{V}$ (see (\ref{eq:rp4})) is bounded (relative to $l^{2}\left(V\right)$)
with bounded inverse if and only if there are constants $A,B$ (depending
on $K,V$) such that 
\[
0<A\leq B<\infty,
\]
and one of the following two conditions hold: 
\begin{enumerate}
\item $A\left\Vert c\right\Vert _{l^{2}\left(V\right)}^{2}\leq\sum_{x\in V}\left|\sum_{y\in V}K\left(x,y\right)c\left(y\right)\right|^{2}\leq B\left\Vert c\right\Vert _{l^{2}\left(V\right)}^{2}$,
for all $c\in l^{2}\left(V\right)$;
\item $A\left\Vert f\right\Vert _{\mathscr{H}_{K}}^{2}\leq\sum_{x\in V}\left|f\left(x\right)\right|^{2}\leq B\left\Vert f\right\Vert _{\mathscr{H}\left(K\right)}^{2}$,
for all \textup{$f\in\mathscr{H}_{K}$.}
\end{enumerate}
\end{prop}

\begin{proof}
From general facts about operators in Hilbert space, we have the equivalence
of the following conditions on a densely defined operator $\mathscr{H}_{1}\xrightarrow{\;T\;}\mathscr{H}_{2}$
acting from one Hilbert space $\mathscr{H}_{1}$ into another. First
$T$ is said to be bounded iff 
\begin{equation}
\left\Vert T\right\Vert _{\mathscr{H}_{1}\rightarrow\mathscr{H}_{2}}=\sup_{u\in\mathscr{H}_{1},\left\Vert u\right\Vert _{1}\leq1}\left\Vert Tu\right\Vert _{2}<\infty.\label{eq:rp6}
\end{equation}
Note that 
\begin{equation}
\left\Vert T^{*}T\right\Vert _{\mathscr{H}_{1}\rightarrow\mathscr{H}_{1}}=\left\Vert TT^{*}\right\Vert _{\mathscr{H}_{2}\rightarrow\mathscr{H}_{2}}=\left\Vert T\right\Vert _{\mathscr{H}_{1}\rightarrow\mathscr{H}_{2}}^{2}.\label{eq:rp7}
\end{equation}

Let $K$, $X$, $V\subset X$ be as specified in the proposition,
and consider 
\begin{equation}
span\left\{ K\left(\cdot,x\right):x\in V\right\} \label{eq:rp8}
\end{equation}
as a dense subspace in $\mathscr{H}\left(K_{V}\right)\coloneqq\text{RKHS}(K|_{V\times V})$.
Then we have a naturally defined operator $\mathscr{H}\left(K_{V}\right)\rightarrow l^{2}\left(V\right)$
with (\ref{eq:rp8}) as its dense domain: Set $T_{V}:\mathscr{H}\left(K_{V}\right)\rightarrow l^{2}\left(V\right)$,
\begin{equation}
T=T_{V}:\mathscr{H}\left(K_{V}\right)\ni f\longmapsto\left(f\left(x\right)\right)_{x\in V}\in l^{2}\left(V\right).\label{eq:rp9}
\end{equation}

Considering the respective inner products in the two Hilbert spaces
$\mathscr{H}\left(K_{V}\right)$ and $l^{2}\left(V\right)$, it follows
that the adjoint operator $T_{V}^{*}:l^{2}\left(V\right)\rightarrow\mathscr{H}\left(K_{V}\right)$,
to (\ref{eq:rp9}) is well defined with dense domain $l_{0}^{2}\left(V\right)$,
and 
\[
\left(T_{V}^{*}\left(c\right)\right)\left(x\right)=\sum_{y\in V}K_{V}\left(x,y\right)c\left(y\right),\quad\forall c\in l_{0}^{2}\left(V\right).
\]

Now the conclusion of the proposition follows once we verify that
the operator 
\begin{equation}
T_{V}T_{V}^{*}:l^{2}\left(V\right)\longrightarrow l^{2}\left(V\right)\label{eq:rp11}
\end{equation}
(possibly unbounded) is given by the $\infty\times\infty$ matrix
$K_{V}\coloneqq K\big|_{V\times V}$. Indeed, for $c\in l_{0}^{2}\left(V\right)$
we have the following computation of $T_{V}T_{V}^{*}$ in (\ref{eq:rp11}):
\begin{equation}
l_{0}^{2}\left(V\right)\ni c\xrightarrow{\quad T_{V}^{*}\quad}\sum_{y\in V}c\left(y\right)K\left(\cdot,y\right)\;\text{(as a function on \ensuremath{V});}\label{eq:rp12}
\end{equation}
and so 
\begin{equation}
\left(T_{V}T_{V}^{*}c\right)\left(x\right)=\sum_{y\in V}K\left(x,y\right)c\left(y\right),\label{eq:rp13}
\end{equation}
which is the desired conclusion. Note that we used that, for fixed
$y$, we have 
\begin{equation}
T_{V}(\underset{\in\mathscr{H}_{K}}{\underbrace{K\left(\cdot,y\right)}})=\left(K\left(x,y\right)\right)_{x\in V}\in l^{2}\left(V\right).\label{eq:rp14}
\end{equation}
See assumption (\ref{eq:rp3a}) above. 
\end{proof}
\begin{cor}
Let $K$, $X$, and $V$ be as specified. Denote by $K_{V}$ the operator
in $l^{2}\left(V\right)$ which have $K\left(\cdot,\cdot\right)\big|_{V\times V}$
as its $\infty\times\infty$ matrix realization, see \propref{rp1}. 

Then there are constants $A$, $B$, $0<A\leq B<\infty$ (depending
on $V$) such that 
\[
A\left\Vert f\right\Vert _{\mathscr{H}\left(K_{V}\right)}^{2}\leq\sum_{x\in V}\left|f\left(x\right)\right|^{2}\leq B\left\Vert f\right\Vert _{\mathscr{H}\left(K_{V}\right)}^{2}
\]
for all $f\in\mathscr{H}\left(K_{V}\right)$, if and only if both
$K_{V}$ and $\left(K_{V}\right)^{-1}$ define bounded operators in
$l^{2}\left(V\right)$. 
\end{cor}

\section{\label{sec:fp}Finite point-configurations governed by p.d. kernels}

Given $X\times X\xrightarrow{\;K\;}\mathbb{R}$ a fixed positive definite
(p.d.) kernel, we consider the case such that for every finite subset
$F\subset X$, the restriction $K_{F}\coloneqq K\left(\cdot,\cdot\right)\big|_{F\times F}$
is invertible. Denote by $K_{F}^{-1}$ the inverse matrix. 
\begin{lem}
Let $P_{F}\coloneqq$ the $\mathscr{H}_{K}$-orthogonal projection
onto 
\begin{equation}
\mathscr{H}_{K}\left(F\right)\coloneqq span\left\{ K\left(\cdot,x\right)\right\} _{x\in F}\label{eq:f1}
\end{equation}
then 
\begin{equation}
\left(P_{F}f\right)\left(\cdot\right)=\sum_{x\in F}K_{F}^{-1}\left(f\big|_{F}\right)_{x}K\left(\cdot,x\right).\label{eq:f2}
\end{equation}
Note $f\big|_{F}\coloneqq\begin{bmatrix}f\left(x_{1}\right) & f\left(x_{2}\right) & \cdots & f\left(x_{n}\right)\end{bmatrix}^{T}$
as a column vector, if $F=\left\{ x_{1},x_{2},\cdots,x_{n}\right\} $,
and so $K_{F}^{-1}\left(f\big|_{F}\right)$ refers to matrix multiplication
\begin{equation}
K_{F}^{-1}\left(f\big|_{F}\right)=\sum_{y\in F}\left(K_{F}^{-1}\right)_{xy}f\left(y\right).\label{eq:f3}
\end{equation}
\end{lem}

\begin{proof}
\textbf{Step 1.} $P_{F}^{*}=P_{F}$. Indeed, 
\begin{eqnarray}
P_{F}^{*} & = & P_{F}\nonumber \\
 & \Updownarrow\nonumber \\
\left\langle K_{F}^{-1}f\big|_{F},g\right\rangle _{\mathscr{H}_{K}} & = & \left\langle f,K_{F}^{-1}g\big|_{F}\right\rangle _{\mathscr{H}_{K}}\nonumber \\
 & \Updownarrow\nonumber \\
\underset{F\times F}{\sum\sum}\left(K_{F}^{-1}\right)_{xy}f\left(y\right)g\left(x\right) & = & \underset{F\times F}{\sum\sum}f\left(x\right)\left(K_{F}^{-1}\right)_{xy}g\left(y\right)\nonumber \\
 & \Updownarrow\nonumber \\
\left(K_{F}^{-1}\right)_{xy} & = & \left(K_{F}^{-1}\right)_{yx}\label{eq:f5}
\end{eqnarray}
but (\ref{eq:f5}) follows from the fact that $K$ is p.d. and so
in particular $K\left(x,y\right)=K\left(y,x\right)$, i.e., symmetric. 
\begin{rem*}
If $K$ is complex valued, the same conclusion holds but with $K\left(x,y\right)=\overline{K\left(y,x\right)}$,
$\forall x,y\in X$.
\end{rem*}
\textbf{Step 2.} $P_{F}^{2}=P_{F}$. Now this follows from the definition
(\ref{eq:f2}) of the finite rank operator $P_{F}:\mathscr{H}_{F}\rightarrow\mathscr{H}_{F}$.
So we iterate the formula (\ref{eq:f2}) which defines the action
of $P_{F}$: 
\begin{eqnarray*}
f & \xrightarrow{\;P_{F}\;} & \sum_{x\in F}\left(K_{F}^{-1}\right)\left(f\big|_{F}\right)_{x}K\left(\cdot,x\right)\\
 & \xrightarrow{\;P_{F}\;} & \underset{F\times F}{\sum\sum}K_{F}^{-1}\left(f\big|_{F}\right)_{x}\underset{\delta_{xy}}{\underbrace{\left(K_{F}^{-1}\right)_{y}K\left(y,x\right)K\left(\cdot,y\right)}}\\
 & = & \underset{F\times F}{\sum\sum}K_{F}^{-1}\left(f\big|_{F}\right)_{x}\delta_{xy}K\left(\cdot,y\right)\\
 & = & \sum_{F}K_{F}^{-1}\left(f\big|_{F}\right)_{x}K\left(\cdot,x\right)\\
 & \underset{\text{by \ensuremath{\left(\ref{eq:f2}\right)}}}{=} & \left(P_{F}f\right)\left(\cdot\right)
\end{eqnarray*}
so $P_{F}^{2}=P_{F}$, and we proved that $P_{F}$ is the desired
finite rank projection in $\mathscr{H}_{K}$ onto the subspace $\mathscr{H}_{K}\left(F\right)$. 
\begin{rem*}
We work with the case that for every $F\subset X$ (discrete points)
\begin{equation}
K_{F}=K\left(\cdot,\cdot\right)\big|_{F\times F}\;\text{is an invertible matrix.}\label{eq:f6}
\end{equation}
But 
\begin{equation}
\left(\ref{eq:f6}\right)\Longleftrightarrow\left\{ K\left(\cdot,x_{1}\right),\cdots,K\left(\cdot,x_{n}\right)\right\} \;\text{is linearly independent},\label{eq:f7}
\end{equation}
where $F=\left(x_{i}\right)$, $x_{i}\neq x_{j}$ if $i\neq j$. To
see this, fix $c=\left(c_{1},\cdots,c_{n}\right)$, then 
\begin{gather*}
\sum_{i}c_{i}K\left(\cdot,x_{i}\right)=0\;\text{in \ensuremath{\mathscr{H}_{K}}}\\
\Updownarrow\\
\left\Vert \sum c_{i}K\left(\cdot,x_{i}\right)\right\Vert _{\mathscr{H}_{K}}^{2}=0\\
\Updownarrow\\
\underset{F\times F}{\sum\sum}c_{i}c_{j}K\left(x_{i},x_{j}\right)=0\\
\Updownarrow\\
c^{T}K_{F}c=0\Longleftrightarrow c=0
\end{gather*}
and so (\ref{eq:f7}) holds, i.e., linear independence.
\end{rem*}
\end{proof}
\begin{cor}
\label{cor:psin}If $F_{0}=\left\{ x_{0}\right\} $ is a singleton,
then 
\[
\left(P_{F_{0}}f\right)\left(\cdot\right)=\frac{f\left(x_{0}\right)}{K\left(x_{0},x_{0}\right)}K\left(\cdot,x_{0}\right),
\]
and 
\[
\underset{\underset{x_{n}}{\downarrow}}{P_{n}}\cdots\underset{\underset{x_{2}}{\downarrow}}{P_{2}}\underset{\underset{x_{1}}{\downarrow}}{P_{1}}\underset{\underset{x_{0}}{\downarrow}}{P_{0}}:f\longmapsto f\left(x_{0}\right)\frac{K\left(x_{0},x_{1}\right)K\left(x_{1},x_{2}\right)\cdots K\left(x_{n},\cdot\right)}{K\left(x_{0},x_{0}\right)K\left(x_{1},x_{1}\right)\cdots K\left(x_{n},x_{n}\right)}.
\]
\end{cor}

\section{\label{sec:PCA}Principal Component Analysis}

In this section we connect the projection operators of Karhunen-Lo\`eve
transform or Principal Component Analysis (PCA) to the projections
operators in sections \ref{sec:fp} and \ref{sec:kac}. These operator
are from \cite{MR2362796,MR2459323,jorgensen2019dimension} where
in computing probabilities and variance, Hilbert spaces was used.
For example, unit vector $f$ was taken in some fixed Hilbert space
$\mathcal{H}$, and an orthonormal basis (ONB) $\psi_{i}$ with $i$
running over an index set $I$. Then two families of probability measures,
were used where one family $P_{f}(\cdot)$ indexed by $f\in\mathcal{H}$,
and a second family $P_{T}$ indexed by a class of operators $T:\mathcal{H}\to\mathcal{H}$.
\begin{defn}
\label{D:H} Let $\mathcal{H}$ be a Hilbert space. Let $(\psi_{i})$
and $(\phi_{i})$ be orthonormal bases (ONB), with index set $I$.
Usually 
\begin{equation}
I=\mathbb{N}=\{1,2,...\}.\label{E:index}
\end{equation}
If $(\psi_{i})_{i\in I}$ is an ONB, we set $Q_{n}:=$ the orthogonal
projection onto $span\{\psi_{1},...,\psi_{n}\}$.
\end{defn}

\begin{prop}
Consider an ensemble of a large number $N$ of objects of similar
type such as a set of data, of which $Nw^{\alpha}$, $\alpha=1,2,...,\nu$
where the relative frequency $w^{\alpha}$ satisfies the probability
axioms: 
\[
w^{\alpha}\geq0,\quad\sum_{\alpha=1}^{\nu}w^{\alpha}=1.
\]

Assume that each type specified by a value of the index $\alpha$
is represented by $f^{\alpha}(\xi)$ in a real domain $[a,b]$, which
we can normalize as
\[
\int_{a}^{b}|f^{\alpha}(\xi)|^{2}d\xi=1.
\]
Let $\{\psi_{i}(\xi)\}$, $i=1,2,...,$ be a complete set of orthonormal
base functions defined on $[a,b]$. Then any function (or data) $f^{\alpha}(\xi)$
can be expanded as 
\begin{equation}
f^{\alpha}(\xi)=\sum_{i=1}^{\infty}x_{i}^{(\alpha)}\psi_{i}(\xi)\label{eq:onbrep}
\end{equation}
with 
\begin{equation}
x_{i}^{\alpha}=\int_{a}^{b}\psi_{i}^{*}\left(\xi\right)f^{\alpha}\left(\xi\right)d\xi.\label{eq:xialph}
\end{equation}
Here, $x_{i}^{\alpha}$ is the component of $f^{\alpha}$ in $\psi_{i}$
coordinate system. With the normalization of $f^{\alpha}$ we have
\[
\sum_{i=1}^{\infty}|x_{i}^{\alpha}|^{2}=1.
\]
\end{prop}

\begin{proof}
If we substitute (\ref{eq:xialph}) into (\ref{eq:onbrep}) we have
\begin{align*}
f^{\alpha}\left(\xi\right) & =\int_{a}^{b}f^{\alpha}\left(\xi\right)\left[\sum\nolimits _{i=1}^{\infty}\psi_{i}^{*}\left(\xi\right)\psi_{i}\left(\xi\right)\right]d\xi\\
 & =\sum\nolimits _{i=1}^{\infty}\left\langle \psi_{i},f^{\alpha}\right\rangle \psi_{i}\left(\xi\right)
\end{align*}
by definition of ONB. Note this involves orthogonal projection.
\end{proof}
We here give mathematical background of PCA.

Let $\mathcal{H}=L^{2}(a,b)$. $\psi_{i}:\mathcal{H}\to\mathit{l}^{2}(\mathbb{Z})$
and $U:\mathit{l}^{2}(\mathbb{Z})\to\mathit{l}^{2}(\mathbb{Z})$ where
$U$ is a unitary operator.

Notice that the distance is invariant under a unitary transformation.
Thus, using another coordinate system (principal axis) $\{\phi_{j}\}$
in place of $\{\psi_{i}\}$, would preserve the distance. The idea
is that when PCA transform is applied on a set of data, the set of
data $\{x_{i}^{\alpha}\}$ in the feature space represented in $\{\psi_{i}\}$
basis are now represented in another coordinate system $\{\phi_{j}\}$
.

Let $\{\phi_{j}\}$, $j=1,2,...,$ be another set of orthonormal basis
(ONB) functions instead of $\{\psi_{i}(\xi)\}$, $i=1,2,...,$. Let
$y_{j}^{\alpha}$ be the component of $f^{\alpha}$ in $\{\phi_{j}\}$
where it can be expressed in terms of $x_{i}^{\alpha}$ by a linear
relation 
\[
y_{j}^{\alpha}=\sum_{i=1}^{\infty}\left\langle \phi_{j},\psi_{i}\right\rangle x_{i}^{\alpha}=\sum_{i=1}^{\infty}U_{i,j}x_{i}^{\alpha}
\]
where $U:\mathit{l}^{2}(\mathbb{Z})\to\mathit{l}^{2}(\mathbb{Z})$
is the unitary operator
\[
U_{i,j}=\langle\phi_{j},\psi_{i}\rangle=\int_{a}^{b}\phi_{j}^{*}\left(\xi\right)\psi_{i}\left(\xi\right)d\xi.
\]
Also, $x_{i}^{\alpha}$ can be written in terms of $y_{j}^{\alpha}$
under the following relation 
\[
x_{i}^{\alpha}=\sum_{j=1}^{\infty}\langle\psi_{i},\phi_{j}\rangle y_{j}^{\alpha}=\sum_{j=1}^{\infty}U_{i,j}^{-1}y_{j}^{\alpha}
\]
where $U_{i,j}^{-1}=\overline{U_{i,j}}$ and $\overline{U_{i,j}}=U_{j,i}^{*}$.
Thus,

\[
f^{\alpha}(\xi)=\sum_{i=1}^{\infty}x_{i}^{\alpha}\left(\xi\right)\psi_{i}\left(\xi\right)=\sum y_{i}^{\alpha}\left(\xi\right)\phi_{i}\left(\xi\right).
\]
So $U(x_{i})=(y_{i})$ which is coordinate change, and $\sum_{i=1}^{\infty}x_{i}^{\alpha}\psi_{i}(\xi)=\sum_{j=1}^{\infty}y_{j}^{\alpha}\phi_{j}(\xi)$,
and

\[
x_{i}^{\alpha}=\left\langle \psi_{i},f^{\alpha}\right\rangle =\int_{a}^{b}\psi_{i}^{*}\left(\xi\right)f^{(\alpha)}\left(\xi\right)d\xi.
\]

The squared magnitude $|x_{i}^{(\alpha)}|^{2}$ of the coefficient
for $\psi_{i}$ in the expansion of $f^{(\alpha)}$ can be considered
as a good measure of the average in the ensemble 
\[
Q_{i}=\sum_{\alpha=1}^{n}w^{(\alpha)}|x_{i}^{(\alpha)}|^{2},
\]
and as a measure of importance of $\{\psi_{i}\}$. Notice, 
\[
Q_{i}\geq0,\text{ }\sum_{i}Q_{i}=1.
\]
See also \cite{Wat65,jorgensen2019dimension}.

Let $G\left(\xi,\xi'\right)=\sum_{\alpha}w^{\alpha}f^{\alpha}\left(\xi\right)f^{\alpha*}\left(\xi'\right)$.
Then $G$ is a Hermitian matrix that is the covariance matrix and
$Q_{i}=G\left(i,i\right)=\sum_{\alpha}w^{\alpha}x_{i}^{\alpha}x_{i}^{\alpha*}$.
Here, $Q_{i}=G\left(i,i\right)$ is the variance and $G\left(i,j\right)$
determines the covariance between $x_{i}$ and $x_{j}$. The normalization
$\sum Q_{i}=1$ gives us $\text{trace }G=1$, where the trace means
the diagonal sum.

Then define a special function system $\{\Theta_{k}(\xi)\}$ as the
set of eigenfunctions of $G$, i.e., 
\begin{equation}
\int_{a}^{b}G\left(\xi,\xi'\right)\Theta_{k}\left(\xi'\right)d\xi'=\lambda_{k}\Theta_{k}(\xi).
\end{equation}
So $G\Theta_{k}(\xi)=\lambda_{k}\Theta_{k}(\xi)$. $Also,U:\mathit{l}^{2}(\mathbb{Z})\to\mathit{l}^{2}(\mathbb{Z})$
is the unitary operator consisting of eigenfunctions of $G$ in its
columns. These eigenfunctions represent the directions of the largest
variance of the data and the corresponding eigenvalues represent the
magnitude of the variance in the directions. PCA allows us to choose
the principal components so that the covariance matrix $G$ of the
projected data is as large as possible. The largest eigenfunction
of the covariance matrix points to the direction of the largest variance
of the data and the magnitude of this function is equal to the corresponding
eigenvalue. The subsequent eigenfunctions are always orthogonal to
the largest eigenfunctions.

\textbf{Principal eigenfunctions and detecting largest variance.}
When the data are not functions but vectors $v^{\alpha}$s whose components
are $x_{i}^{(\alpha)}$ in the $\psi_{i}$ coordinate system, we have
\begin{equation}
\sum_{i'}G\left(i,i'\right)t_{i'}^{k}=\lambda_{k}t_{i}^{k}
\end{equation}
where $t_{i}^{k}$ is the $i^{th}$ component of the vector $\Theta_{k}$
in the coordinate system $\{\psi_{i}\}$. So we get $\psi:\mathcal{H}\to(x_{i})$
and also $\Theta:\mathcal{H}\to(t_{i})$. The two ONBs result in 
\[
x_{i}^{\alpha}=\sum_{k}c_{k}^{\alpha}t_{i}^{k}\text{ for all }i,\quad c_{k}^{\alpha}=\sum_{i}t_{i}^{k*}x_{i}^{\alpha},
\]
which is the Karhunen-Lo\`eve expansion of $f^{\alpha}\left(\xi\right)$
or vector $v^{\alpha}$. Hence $\{\Theta_{k}(\xi)\}$ is the K-L coordinate
system dependent on $\{w^{\alpha}\}$ and $\{f^{\alpha}(\xi)\}$.
Then we arrange the corresponding eigenfunctions or eigenvectors in
the order of eigenvalues $\lambda_{1}\geq\lambda_{2}\geq\ldots\geq\lambda_{k-1}\geq\lambda_{k}\geq\ldots$
in the columns of $U$.

Now, $Q_{i}=G_{i,i}=\langle\psi_{i},G\psi_{i}\rangle=\sum_{k}A_{ik}\lambda_{k}$
where $A_{ik}=t_{i}^{k}t_{i}^{k*}$ which is a double stochastic matrix.
Then we have the following eigendecomposition of the covariance matrix
(operator), $G$
\begin{equation}
G=U\begin{pmatrix}\lambda_{1} & \cdots & 0\\
0 & \ddots & 0\\
0 & \cdots & \lambda_{k}
\end{pmatrix}U^{-1}.\label{eq:3.9}
\end{equation}

\subsection{\label{subsec:KLEV}Principal Component Analysis and Maximal Variance}

In this subsection, we discuss the orthonormal bases of Karhunen-Lo\`eve
transform or PCA where it captures the maximal variance in the linear
data to effectively perform dimensionality reduction. In \cite{jorgensen2019dimension}
these results to our Principal Component Analysis (PCA) on data, and
dimension reduction algorithms for both linear and nonlinear data
sets were shown. We shall recall here some definitions and results
from \cite{MR1913212,MR2362796,jorgensen2019dimension}.

The following definitions, lemmas and theorem are results from \cite{jorgensen2019dimension,MR2362796}.
Let $\mathcal{H}$ be a Hilbert space which realizes trace class $G$
as a self-adjoint operator.
\begin{defn}
\label{D:traceclass} $T\in B\left(\mathcal{H}\right)$ is said to
be trace class if and only if the series $\sum\left\langle \psi_{i},\left|T\right|\psi_{i}\right\rangle $,
with $\left|T\right|=\sqrt{T^{*}T}$, is convergent for some ONB $\left(\psi_{i}\right)$.
In this case, set 
\begin{equation}
tr\left(T\right):=\sum\left\langle \psi_{i},T\psi_{i}\right\rangle .\label{E:trace}
\end{equation}
\end{defn}

\begin{defn}
\label{def:frame}A sequence $\left(h_{\alpha}\right)_{\alpha\in A}$
in $\mathcal{H}$ is called a \emph{frame} if there are constants
$0<c_{1}\leq c_{2}<\infty$ such that

\begin{equation}
c_{1}\left\Vert f\right\Vert ^{2}\leq\sum_{\alpha\in A}\left|\left\langle h_{\alpha},f\right\rangle \right|^{2}\leq c_{2}\left\Vert f\right\Vert ^{2}\text{ for all }f\in\mathcal{H}.\label{E:framebd}
\end{equation}
Also see \cite{MR1913212,MR2367342,MR2193805,MR3928472,MR3441732,MR3085820,jorgensen2019dimension}.
\end{defn}

\begin{lem}
\label{lem:L1} Let $\left(h_{\alpha}\right)_{\alpha\in A}$ be a
frame in $\mathcal{H}$. Set $L:\mathcal{H}\to\mathit{l}^{2}$, 
\begin{equation}
L:f\mapsto\left(\left\langle h_{\alpha},f\right\rangle \right)_{\alpha\in A}.\label{E:map_L}
\end{equation}
Then $L^{*}:\mathit{l}^{2}\to\mathcal{H}$ is given by 
\begin{equation}
L^{*}((c_{\alpha}))=\sum_{\alpha\in A}c_{\alpha}h_{\alpha}\label{E:L*}
\end{equation}
where $(c_{\alpha})\in\mathit{l}^{2}$; and 
\begin{equation}
L^{*}L=\sum_{\alpha\in A}\left|h_{\alpha}\left\rangle \right\langle h_{\alpha}\right|.\label{E:L*L}
\end{equation}
\end{lem}

\begin{defn}
\label{def:frameop}Suppose we are given $(f_{\alpha})_{\alpha\in A}$,
a frame, non-negative numbers $\{w_{\alpha}\}_{\alpha\in A}$, where
$A$ is an index set, with $\|f_{\alpha}\|=1$, for all $\alpha\in A$.
\begin{equation}
G:=\sum_{\alpha\in A}w_{\alpha}\left|f_{\alpha}\left\rangle \right\langle f_{\alpha}\right|\label{eq:frame_operator}
\end{equation}
is called a \textbf{frame} operator associated to $(f_{\alpha})$.
\end{defn}

\begin{rem}
If we take vectors $(f_{\alpha})$ from a frame $(h_{\alpha})$ and
normalize them such that $h_{\alpha}=\left\Vert h_{\alpha}\right\Vert f_{\alpha}$,
and $w_{\alpha}:=\left\Vert h_{\alpha}\right\Vert ^{2}$, then $L^{*}L$
has the form (\ref{eq:frame_operator}) and it becomes to covariance
matrix G above. Thus $G=L^{*}L:\mathcal{H\to\mathcal{H}}$.
\end{rem}

\begin{lem}
\label{lem:trG}Let $G$ be as in (\ref{eq:frame_operator}). Then
$G$ is trace class if and only if $\sum_{\alpha}w_{\alpha}<\infty$;
and then 
\begin{equation}
tr\left(G\right)=\sum_{\alpha\in A}w_{\alpha}.\label{E:trG}
\end{equation}
\end{lem}

\begin{defn}
\label{D:G-op} Suppose we are given a frame operator 
\begin{equation}
G=\sum_{\alpha\in A}w_{\alpha}\left|f_{\alpha}\left\rangle \right\langle f_{\alpha}\right|\label{eq:g_op}
\end{equation}
and an ONB $(\psi_{i})$. Then for each $n$, the numbers 
\begin{equation}
E_{n}^{\psi}=\sum_{\alpha\in A}w_{\alpha}\|f_{\alpha}-\sum_{i=1}^{n}\left\langle \psi_{i},f_{\alpha}\right\rangle \psi_{i}\|^{2}\label{eq:errorterm}
\end{equation}
are called the \emph{error} or the \textit{residual} of the projection.
\end{defn}

\begin{lem}
\label{lem:error}When $(\psi_{i})$ is given, set $Q_{n}:=\sum_{i=1}^{n}\left|\psi_{i}\left\rangle \right\langle \psi_{i}\right|$
and $Q_{n}^{\bot}=I-Q_{n}$ where $I$ is the identity operator in
$\mathcal{H}$. Then (see (\ref{eq:errorterm})) 
\begin{equation}
E_{n}^{\psi}=tr(GQ_{n}^{\bot}).\label{E:error}
\end{equation}
\end{lem}

The more general frame operators are as follows: Let 
\begin{equation}
G=\sum_{\alpha\in A}w_{\alpha}P_{\alpha}\label{eq:gen_fram_op}
\end{equation}
where $(P_{\alpha})$ is an indexed family of projection operators
in $\mathcal{H}$, i.e., $P_{\alpha}=P_{\alpha}^{*}=P_{\alpha}^{2}$,
for all $\alpha\in A$. $P_{\alpha}$ is trace class if and only if
it is finite-dimensional, i.e., if and only if the subspace $P_{\alpha}\mathcal{H}=\left\{ x\in\mathcal{H}\mid P_{\alpha}x=x\right\} $
is finite-dimensional.

PCA, is the scheme involves a choice of \textquotedblleft principal
components,\textquotedblright{} often realized as a finite-dimensional
subspace of a global (called latent) data set. There are two views
of principal components: The simplest case of consideration of covariance
operators, and in \cite{jorgensen2019dimension} kernel PCA which
refers to a class of reproducing kernels, as used in learning theory
is discussed. In the latter case, one identifies principal features
for the machine learning algorithm.

In can be observed that the simplest way to identify a PCA subspace
is to turn to a covariance operator, namely $G$, acting on the global
data. With the use of a suitable Karhunen-Lo\`eve transform or PCA,
and via a system of i.i.d. standard Gaussians, a covariance operator
which is of trace class may be obtained. An application of the spectral
theorem to this associated operator $G$ (see (\ref{eq:G1}) below),
we then the algorithm for computing eigenspaces corresponding to the
top of the spectrum of $G$, i.e., the subspace spanned by the eigenvectors
for the top $n$ eigenvalues can be obtained; see (\ref{eq:G2}).
These subspaces will then be principal components of order $n$ since
the contribution from the span of the remaining eigenspaces will be
negligible. The algorithm and example will be given in the next subsections.

A second approach to PCA is based on an analogous identification of
principal component subspaces, but with the optimization involving
maximum likelihood, or minimization of ``cost.\textquotedblright{}

Now, although PCA is used popularly in linear data dimension reductions
as PCA decorrelates data, it is noted that the decorrelation only
corresponds to statistical independence in the Gaussian case. So PCA
is not generally the optimal choice for linear data dimension reduction.
However, \textcolor{black}{PCA captures maximal variability in the
data. The reader may find more details in }\cite{jorgensen2019dimension,5158580,7005973}.

PCA enables finding projections which maximize the variance: The first
principal component is the direction in the feature space along which
gives projections with the largest variance. The second principal
component is the variance maximizing direction to all directions orthogonal
to the first principal component. The $i^{th}$ component is the direction
which maximizes variance orthogonal to the $i-1$ previous components.
Thus, PCA captures maximal variability then projects a set of data
in higher dimensional feature space to a lower dimensional feature
space orthogonally and this was proved in Theorem 2.12 in \cite{jorgensen2019dimension}
which is the following theorem. Below, we formulate the iterative
algorithm (see eq. (\ref{eq:G2})) of producing principal components
in the context of trace class operators. In the proof of the theorem,
we can observe how the orthogonal projection operators play the key
part in capturing the maximal variance for PCA application. The example
will be shown in the next subsections with a matrix and with PCA image
compression with principal components which are obtained by using
the projection operators.
\begin{thm}
\label{thm:smallest_error} The Karhunen-Lo\`eve ONB with respect
to the frame operator $G=L^{*}L$ gives the smallest error in the
approximation to a frame operator and the covariance operator $G$
gives maximum variance.
\end{thm}

In the proof of this theorem in \cite{jorgensen2019dimension}, we
use the covariance operator $G$ which is trace class and positive
semidefinite, applying the spectral theorem to $G$ results is a discrete
spectrum, with the natural order $\lambda_{1}\geq\lambda_{2}\geq...$
and a corresponding ONB $(\phi_{k})$ consisting of eigenvectors,
i.e., 
\begin{equation}
G\phi_{k}=\lambda_{k}\phi_{k},\;k\in\mathbb{N},\label{eq:G1}
\end{equation}
called the Karhunen-Lo\`eve data or principal components. The spectral
data is constructed recursively starting with
\begin{equation}
\begin{split}\lambda_{1} & =\sup_{\phi\in\mathcal{H},\:\left\Vert \phi\right\Vert =1}\left\langle \phi,G\phi\right\rangle =\left\langle \phi_{1},G\phi_{1}\right\rangle ,\;\text{and}\\
\lambda_{k+1} & =\sup_{\stackrel{\phi\in\mathcal{H},\:\left\Vert \phi\right\Vert =1}{\phi\perp\phi_{1},\phi_{2},\dots,\phi_{k}}}\left\langle \phi,G\phi\right\rangle =\left\langle \phi_{k+1},G\phi_{k+1}\right\rangle .
\end{split}
\label{eq:G2}
\end{equation}
This way, the maximal variance is achieved. Now by applying \cite{ArKa06}
we have
\begin{equation}
\sum_{k=1}^{n}\lambda_{k}\geq tr\left(Q_{n}^{\psi}G\right)=\sum_{k=1}^{n}\left\langle \psi_{k},G\psi_{k}\right\rangle \quad\text{for all }n,\label{eq:ineq}
\end{equation}
where $Q_{n}^{\psi}$ is the sequence of projections, deriving from
some ONB $(\psi_{i})$ and are arranged such that the following holds:
\[
\left\langle \psi_{1},G\psi_{1}\right\rangle \geq\left\langle \psi_{2},G\psi_{2}\right\rangle \geq...\quad\text{.}
\]
Hence we are comparing ordered sequences of eigenvalues with sequences
of diagonal matrix entries. So here, \textbf{the sequence of projections
$Q_{n}^{\psi}$ derived from deriving from some ONB $(\psi_{i})$
play the key role in selection of principal components.}

Lastly, we have 
\[
tr\left(G\right)=\sum_{k=1}^{\infty}\lambda_{k}=\sum_{k=1}^{\infty}\left\langle \psi_{k},G\psi_{k}\right\rangle <\infty.
\]

The assertion in \thmref{smallest_error} is the validity of 
\begin{equation}
E_{n}^{\phi}\leq E_{n}^{\psi}\label{eq:error_ineq}
\end{equation}
for all $(\psi_{i})\in ONB(\mathcal{H})$, and all $n=1,2,...$; and
moreover, that the infimum on the RHS in (\ref{eq:error_ineq}) is
attained for the KL-ONB $(\phi_{k})$. But in view of our lemma for
$E_{n}^{\psi}$ (\ref{lem:error}), we see that (\ref{eq:error_ineq})
is equivalent to the system (\ref{eq:ineq}) in the Arveson-Kadison
theorem.

\subsection{\label{subsec:AlgI}The Algorithm for a Digital Image Application}

Here, we show how the above PCA theory with orthogonal projections
\textbf{$Q_{n}^{\psi}$} are implemented as algorithm for application.
Our aim is to reduce the number of bits needed to represent an image
by removing redundancies as much as possible. Karhunen-Loève transform
or PCA is a transform of $m$ vectors with the length $n$ formed
into $m$-dimensional vector $X=[X_{1},\cdots,X_{m}]$ into a vector
$Y$ according to 
\begin{equation}
Y=A\left(X-m_{X}\right),\label{eq:3.1.1}
\end{equation}
where matrix $A$ is obtained by eigenvectors of the covariance matrix
$C$ as in (\ref{eq:mc}) below.

The algorithm for Karhunen-Loève transform or PCA can be described
as follows: 
\begin{itemize}
\item[1.] Take an image or data matrix $X$, and compute the mean of the column
vectors of $X$ 
\begin{equation}
m_{X}=E\left(X\right)=\frac{1}{n}\sum_{i=1}^{n}X_{i}.\label{eq:meanX}
\end{equation}
\item[2.] Subtract the mean: Subtract the mean, $m_{X}$ in (\ref{eq:meanX})
from each column vector of $X$. This produces a data set matrix $B$
whose mean is zero, and it is called centering the data. 
\item[3.] Compute the covariance matrix from the matrix in the previous step
\begin{align}
C=cov(X) & =E\left(\left(X-m_{X}\right)\left(X-m_{X}\right)^{T}\right)\label{eq:mc}\\
 & =\frac{1}{n}\sum_{i=1}^{n}X_{i}X_{i}^{T}-m_{X}m_{X}^{T}.
\end{align}
Here $X-m_{X}$ can be interpreted as subtracting $m_{X}$ from each
column of $X$. $C\left(i,i\right)$ lying in the main diagonal are
the variances of 
\begin{equation}
C(i,i)=E((X_{i}-{m_{X_{i}}})^{2}).
\end{equation}
Also, $C(i,j)=E\left(\left(X_{i}-m_{X_{i}}\right)\left(X_{j}-m_{X_{j}}\right)\right)$
is the covariance between $X_{i}$ and $X_{j}$. 
\item[4.] Compute the eigenvectors and eigenvalues, $\lambda_{i}$ of the covariance
matrix. 
\item[5.] Choose components and form a feature vector (matrix of vectors),
\begin{equation}
A=(eig_{1},...,eig_{n}).
\end{equation}
List the eigenvectors in decreasing order of the magnitude of their
eigenvalues. This matrix $A$ is called the row feature matrix. By
normalizing the column vectors of matrix A, this new matrix $P$ becomes
an orthogonal matrix. Eigenvalues found in step 4 are different in
values. The eigenvector with highest eigenvalue is the principle component
of the data set. Here, the eigenvectors of eigenvalues that are not
up to certain specific values can be dropped thus creating a data
matrix with less dimension value. 
\item[6.] Derive the new data set. 
\[
\text{Final Data}=\text{Row Feature Matrix}\times\text{Row Data Adjust}.
\]
\end{itemize}
The rows of the feature matrix $A$ are orthogonal so the inversion
of PCA can be done on equation (\ref{eq:3.1.1}) by

\begin{equation}
X=A^{T}Y+m_{X}.\label{eq:3.1.2}
\end{equation}

With the $l$ largest eigenvalues with more variance are used instead
of $n$eigenvalues, the matrix $A_{l}$ is formed using the $l$ corresponding
eigenvectors. This yields the newly constructed data or image $X'$
as follows:

\begin{equation}
X'=A_{l}^{T}Y+m_{X}.\label{eq:3.1.2-1}
\end{equation}

Row Feature Matrix is the matrix that has the eigenvectors in its
rows with the most significant eigenvector (i.e., with the greatest
eigenvalue) at the top row of the matrix. Row Data Adjust is the matrix
with mean-adjusted data transposed. That is, the matrix contains the
data items in each column with each row having a separate dimension
(see e.g., \cite{MuPr05,MR3097610,Mar14,jorgensen2019dimension,AT20,JR20,7005973,5158580,Jolliffe2016PrincipalCA}).

In PCA, image compression occurs by the method of dimension reduction.
Here, we need to determine how to choose the right axes. PCA gives
a linear subspace of dimension that is lower than the dimension of
the original image data in such a way that the image data points lie
mainly in the linear subspace with the lower dimension. PCA creates
a new feature-space (subspace) that captures as much variance in the
original image data as possible. The linear subspace is spanned by
the orthogonal vectors that form a basis. These orthogonal vectors
give principal axes, i.e., directions in the data with the largest
variations. As in section \ref{subsec:AlgI}, the PCA algorithm performs
the centering of the image data by subtracting off the mean, and then
determines the direction with the largest variation of the data and
chooses an axis in that direction, and then further explores the remaining
variation and locates another axis that is orthogonal to the first
and explores as much of the remaining variation as possible. This
iteration is performed until all possible axes are exhausted. Once
we have a principal axis, we subtract the variance along this principal
axis to obtain the remaining variance. Then the same procedure is
applied again to obtain the next principal axis from the residual
variance. In addition to being the direction of maximum variance,
the next principal axis must be orthogonal to the other principal
axes. When all the principal axes are obtained, the data set is projected
onto these axes. These new orthogonal coordinate axes are also called
principal components.

The outcome is all the variation along the axes of the coordinate
set, and this makes the covariance matrix diagonal which means each
new variable is uncorrelated with the rest of the variables except
itself. As for some of the axes that are obtained towards last have
very little variation. So they don't contribute much, thus, can be
discarded without affecting the variability in the image data, hence
reducing the dimension (see e.g., \cite{Mar14}).

In PCA image compression, feature selection method is used where we
go through the available features of an image and select useful features
such as variables or predictors, i.e., \textbf{correlation of pixel
values} to the output variables.

PCA removes redundancies and describe the image data with less properties
in a way that it performs a linear transformation moving the original
image data to a new space spanned by principal component. This done
by constructing a new set of properties based on combination of the
old properties. The properties that present low variance are considered
not useful. PCA looks for properties that has maximal variation across
the data to make the principal component space. The eigenvectors found
in PCA algorithm are the new set of axes of the principal component.
Dimension reduction occurs when the eigenvectors with more variance
are chosen but those with less variance are discarded. \cite{Mar14,MuPr05,jorgensen2019dimension,5158580,7005973,Jolliffe2016PrincipalCA}

\subsection{\label{subsec:img}Principal Component Analysis in a Digital Image}

We would like to use a color digital image PCA to illustrate dimension
change in this section, so we introduce a color digital image. A color
digital image is read into a matrix of pixels. We would like to use
Karhunen-Loève transform or PCA applied to a digital image data illustrate
dimension reduction. Here, an image is represented as a matrix of
functions where the entries are pixel values. The following is an
example of a matrix representation of a digital image: 
\begin{equation}
\mathbf{f(x,y)}=\begin{pmatrix}f(0,0) & f(0,1) & \cdots & f(0,N-1)\\
f(1,0) & f(1,1) & \cdots & f(1,N-1)\\
\vdots & \vdots & \vdots & \vdots\\
f(M-1,0) & f(M-1,1) & \cdots & f(M-1,N-1)
\end{pmatrix}.\label{eq:imagematrix-2}
\end{equation}
A color image has three components. Thus a color image matrix has
three of above image pixel matrices for red, green and blue components
and they all appear black and white when viewed ``individually.''
We begin with the following duality principle, (i) \emph{spatial}
vs (ii) \emph{spectral}, and we illustrate its role for the redundancy,
and for correlation of variables, in the resolution-refinement algorithm
for \emph{images}. Specifically: 
\begin{itemize}
\item[(i)] \emph{Spatial Redundancy: }correlation between neighboring pixel
values. 
\item[(ii)] \emph{Spectral Redundancy}: correlation between different color planes
or spectral bands. 
\end{itemize}
We are interested in removing these redundancies using correlations.

Starting with a matrix representation for a particular image, we then
compute the covariance matrix using the steps from (3) and (4) in
algorithm above. We then compute the Karhunen-Loève eigenvalues. Next,
the eigenvalues are arranged in decreasing order. The corresponding
eigenvectors are arranged to match the eigenvalues with multiplicity.
The eigenvalues mentioned here are the same eigenvalues $\lambda_{i}$
in step 4 above, thus yielding smallest error and smallest entropy
in the computation (see e.g., \cite{MR2459323,jorgensen2019dimension,5158580,7005973,Jolliffe2016PrincipalCA}).

The following figure shows the principal components of an image in
increasing eigenvalues where the original image is a color png file.

\begin{figure}[htb]
\centering{}\includegraphics[width=2in]{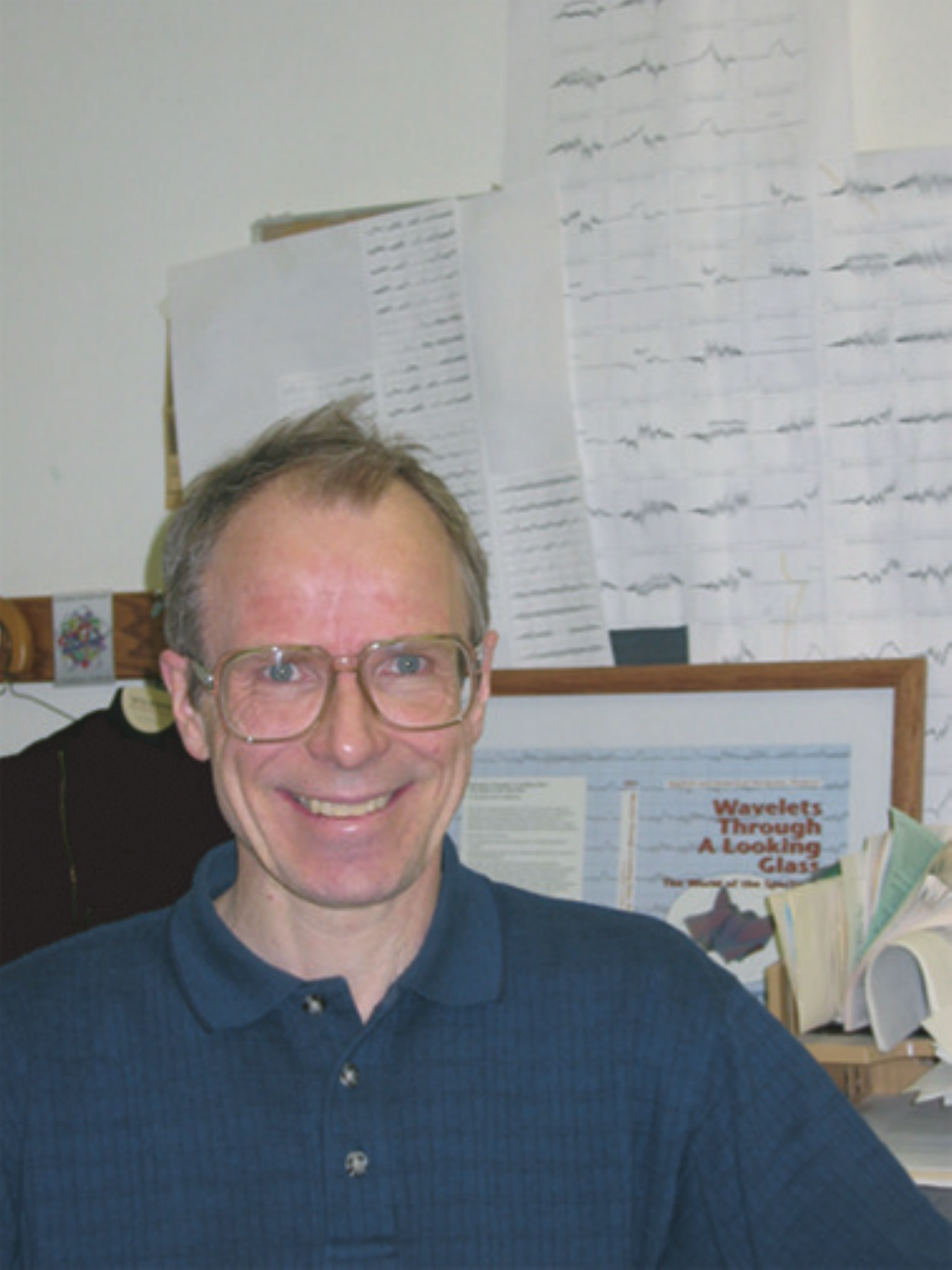} \caption{The Original Jorgensen Image}
\label{PCA_Jor}
\end{figure}

\begin{figure}[h]
\centering \includegraphics[width=0.32\linewidth]{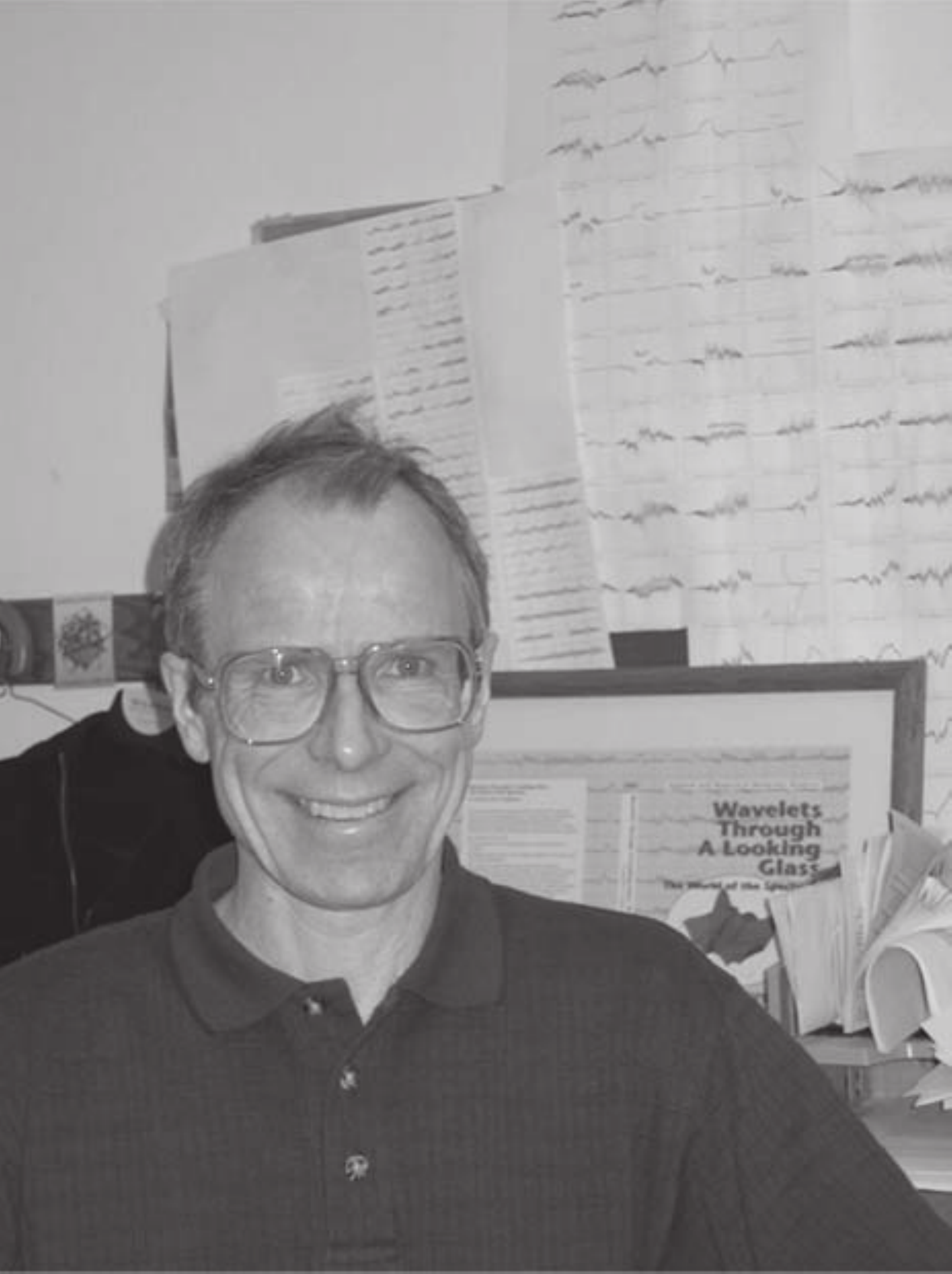} \includegraphics[width=0.32\linewidth]{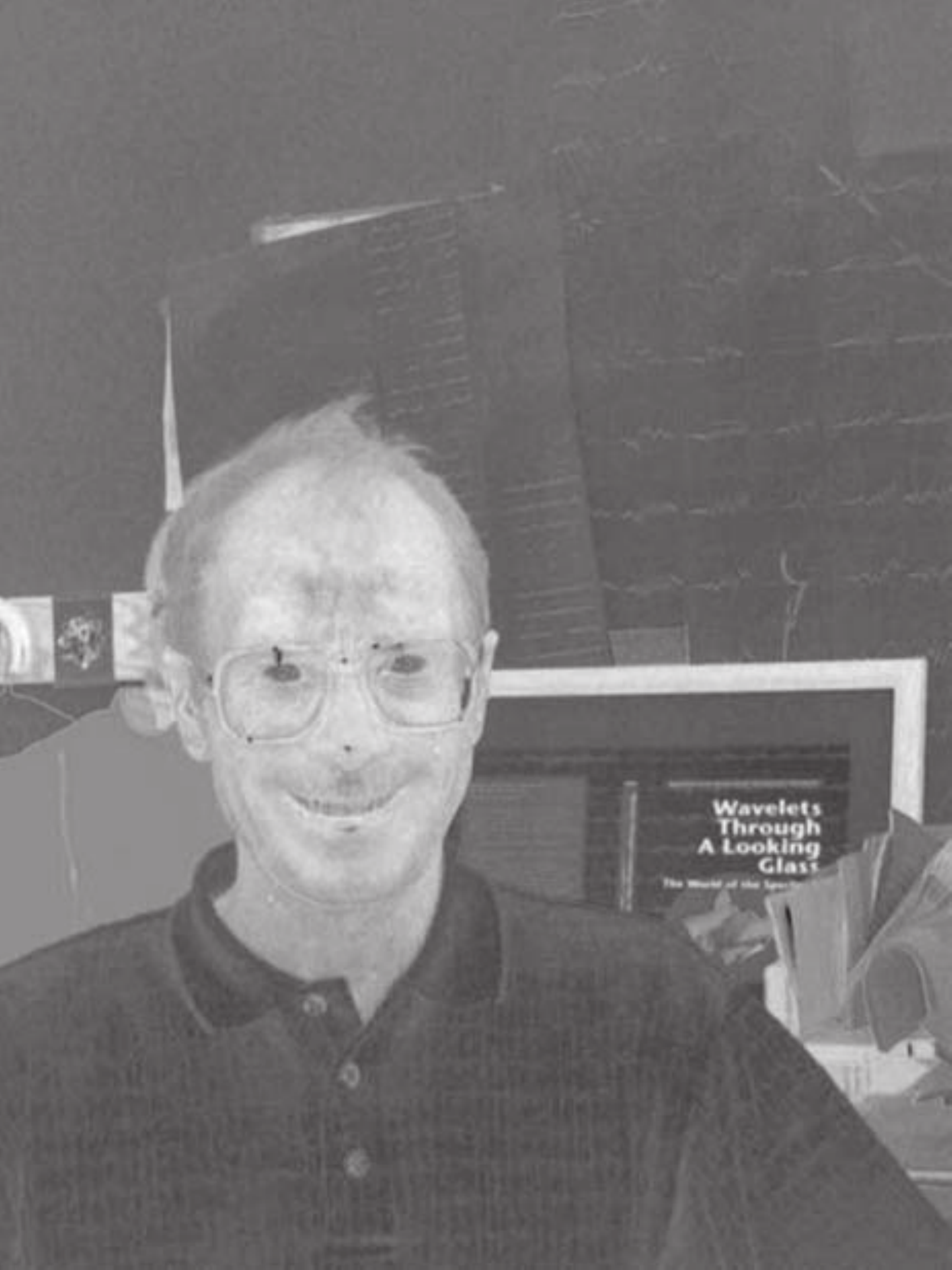}
\includegraphics[width=0.32\linewidth]{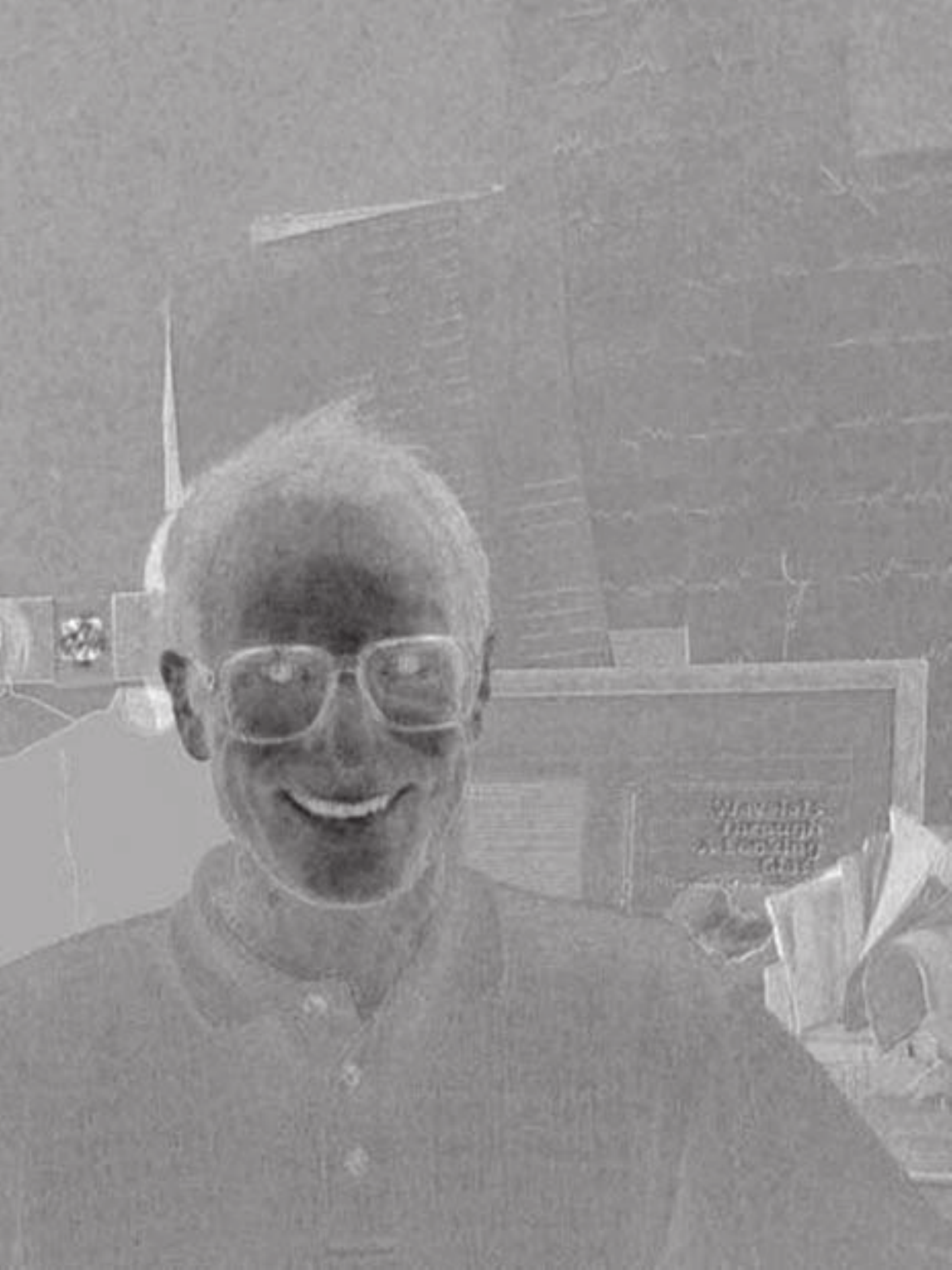}\ \caption{From left to right: First principal component of Jorgensen which corresponds
to the largest eigenvalue, second principal component of Jorgensen
which corresponds to the second largest eigenvalue, third principal
component of Jorgensen which corresponds to the third largest eigenvalue.}
\label{PCA_Jor1}
\end{figure}

The original file is in red, green and blue color image which had
three R, G, B color components. So if $I$ is the original image it
can be represented as 
\begin{equation}
I=w_{1}R+w_{2}G+w_{3}B=f_{R}+f_{G}+f_{B}
\end{equation}
where $w_{1}$, $w_{2}$ and $w_{3}$ are weights which are determined
for different light intensity for color, and $f_{R}$, $f_{G}$ and
$f_{B}$ are the three R, G, B components of the form equation (\ref{eq:imagematrix-2}).
Each matrix appears black and white when viewed individually. Now,
when PCA is performed on the image $I$, it gives alternative components.
Here the original image $I$ is \ref{PCA_Jor}. The original image
used for \ref{PCA_Jor}, is in red, green and blue color components
which are $f_{R}$, $f_{G}$, and $f_{B}$. Please see \cite{jorgensen2019dimension,MuPr05}.

Here, after PCA transformation, instead of RGB components a new three
components are used and this is shown in the Figures \ref{PCA_Jor1}.
The principal components of the image are in the order of increasing
eigenvalues.

Once we see the principal components of the image, we can notice the
significance of the eigenvalue $\lambda$. With PCA, the principal
component corresponding to the largest eigenvalue picks up all the
dominant features of the original image in comparison to the second
principal component corresponding to the second largest eigenvalue
and so on. So the more prominent features of the image is captured
in the principal components in the order of decreasing eigenvalues
with the correlation of pixels of images.

\subsection{A Matrix Example}

In this section, we provide a matrix example of PCA to illustrated
how the PCA algorithm in section \ref{subsec:AlgI} works.

Let, 
\begin{equation}
X=\begin{bmatrix}1 & 0 & 0 & 3\\
-1 & 1 & 1 & 0\\
-1 & -2 & 4 & -5\\
0 & 3 & -1 & 0
\end{bmatrix}\text{ be an image matrix.}
\end{equation}

Following the algorithm of section \ref{subsec:AlgI}, we compute
the mean of each column of the above matrix, $X$. Then we subtract
the mean of each column. For matrix $X$, column 1 has mean of -0.25,
column 2 has mean of 0.5, column 3 has mean of 1, and column 4 has
mean of -0.5.

Next we compute the covariance matrix of $X$ namely $C=cov(X)$.
\begin{equation}
C=cov(X)=\begin{bmatrix}0.9167 & 0.5000 & -1.3333 & 2.5000\\
0.5000 & 4.3333 & -4.0000 & 3.6667\\
-1.3333 & -4.0000 & 4.6667 & -6.0000\\
2.5000 & 3.6667 & -6.0000 & 11.0000
\end{bmatrix}
\end{equation}

The eigenvalues of the covariance matrix, $C=cov(X)$ are $\lambda_{1}=0$,
$\lambda_{2}=0.3551$, $\lambda_{3}=3.2692$, and $\lambda_{4}=17.2924$.
The corresponding eigenvectors of the covariance matrix, $C=cov(X)$
are as follows:

For $\lambda_{1}=0$, the corresponding eigenvector is 
\[
v_{1}=\begin{bmatrix}0.4295\\
0.5154\\
0.7302\\
0.1289
\end{bmatrix}
\]

For $\lambda_{2}=0.3551$, the corresponding eigenvector is 
\[
v_{2}=\begin{bmatrix}0.8584\\
-0.1339\\
-0.3484\\
-0.3519
\end{bmatrix}
\]

For $\lambda_{3}=3.2692$, the corresponding eigenvector is 
\[
v_{3}=\begin{bmatrix}0.2240\\
-0.7582\\
0.3102\\
0.5279
\end{bmatrix}
\]

For $\lambda_{4}=17.2924$, the corresponding eigenvector is 
\[
v_{4}=\begin{bmatrix}0.1685\\
0.3762\\
-0.4992\\
0.7622
\end{bmatrix}
\]

We then form a matrix $A$ which consists of eigenvectors in the columns
of the matrix. The eigenvectors point to the direction of the principal
components represented by eigenvectors with magnitude of eigenvalues.
These $\lambda$ eigenvalues are the variance of the image data and
the magnitude of the eigenvector direction or principal component
direction. The first column of matrix $A$ is the eigenvector $v_{4}$
corresponding to the largest eigenvalue $\lambda_{4}$ and the column
of matrix $A$ is the eigenvector $v_{3}$ corresponding to the largest
eigenvalue $\lambda_{3}$ and so on in decreasing order of eigenvalues
from the largest to smallest respectively. 
\begin{equation}
A=\begin{bmatrix}0.1685 & 0.2240 & 0.8584 & 0.4295\\
0.3762 & -0.7582 & -0.1339 & 0.5154\\
-0.4992 & 0.3102 & -0.3484 & 0.7302\\
0.7622 & 0.5279 & -0.3519 & 0.1289
\end{bmatrix}
\end{equation}

Then we can put the corresponding eigenvalues as diagonal entries
in the $D$ matrix as follows, 
\begin{eqnarray}
D=\left[\begin{array}{cccc}
17.2924 & 0 & 0 & 0\\
0 & 3.2692 & 0 & 0\\
0 & 0 & 0.3551 & 0\\
0 & 0 & 0 & 0
\end{array}\right]
\end{eqnarray}
which is the diagonal matrix of eigenvalues. \cite{MR2362796}

Performing Principal Component Analysis on our matrix $X$ will result
in putting the eigenvectors of $C=cov(X)$ in column vector form in
decreasing order of eigenvalue. After removing the least significant
eigenvector $v_{4}$ we obtain $A_{l}$ matrix 
\begin{equation}
A_{l}=\begin{bmatrix}0.1685 & 0.2240 & 0.8584\\
0.3762 & -0.7582 & -0.1339\\
-0.4992 & 0.3102 & -0.3484\\
0.7622 & 0.5279 & -0.3519
\end{bmatrix}
\end{equation}

Then using equation (\ref{eq:3.1.2-1}) we can obtain compressed image
$X'$.

\subsection{Digital Image Compression Using Principal Component Analysis}

In this section, we will do PCA image compression using black and
white Barbara image using different number of principal components
to reconstruct the image. For more information please see, \cite{AT20,JR20}.

\begin{figure}[htb]
\centering{}\includegraphics[width=2in]{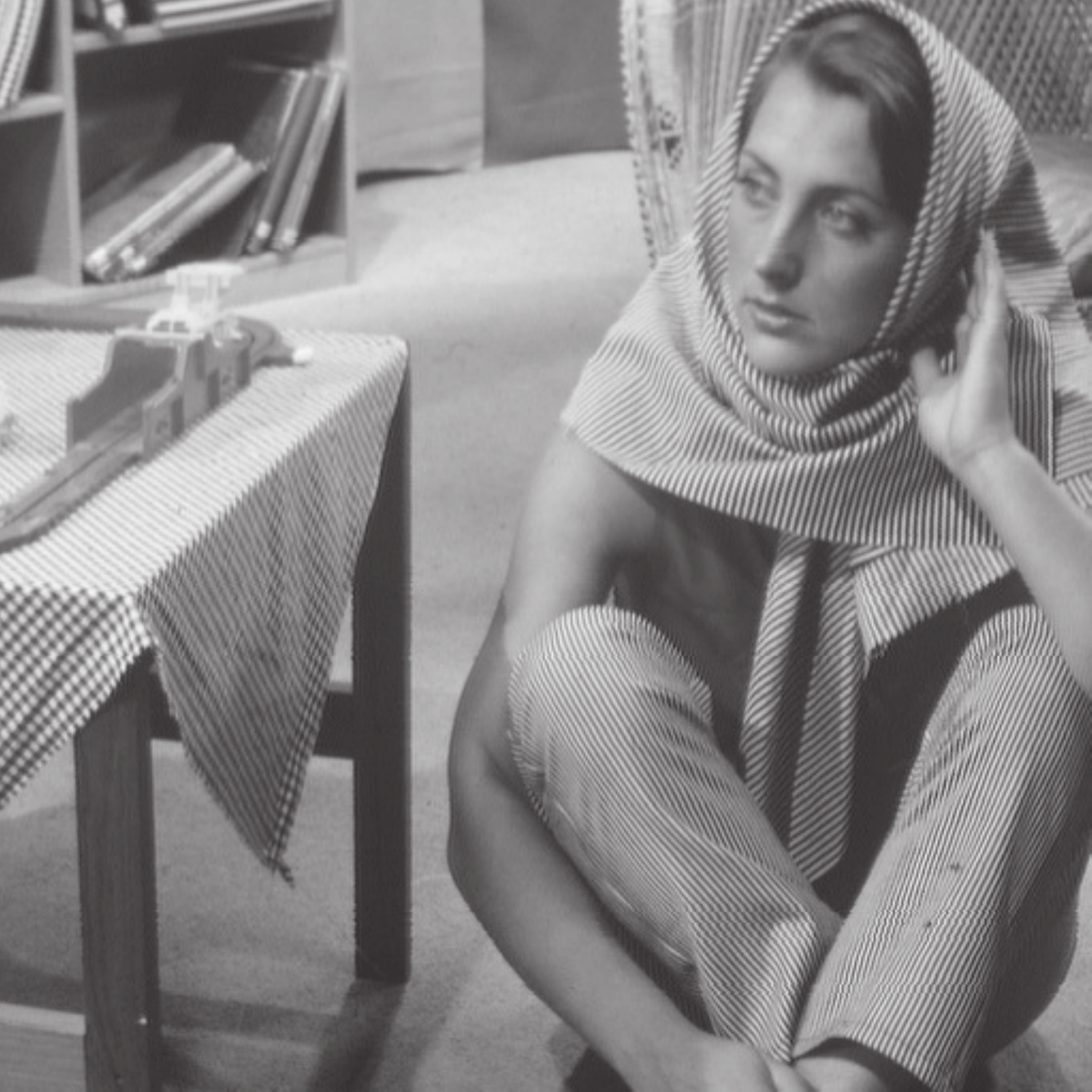} \caption{The original Barbara image from the Test image pool in Google.}
\label{Barbara}
\end{figure}

\begin{figure}[h]
\centering 
 \includegraphics[width=0.32\linewidth]{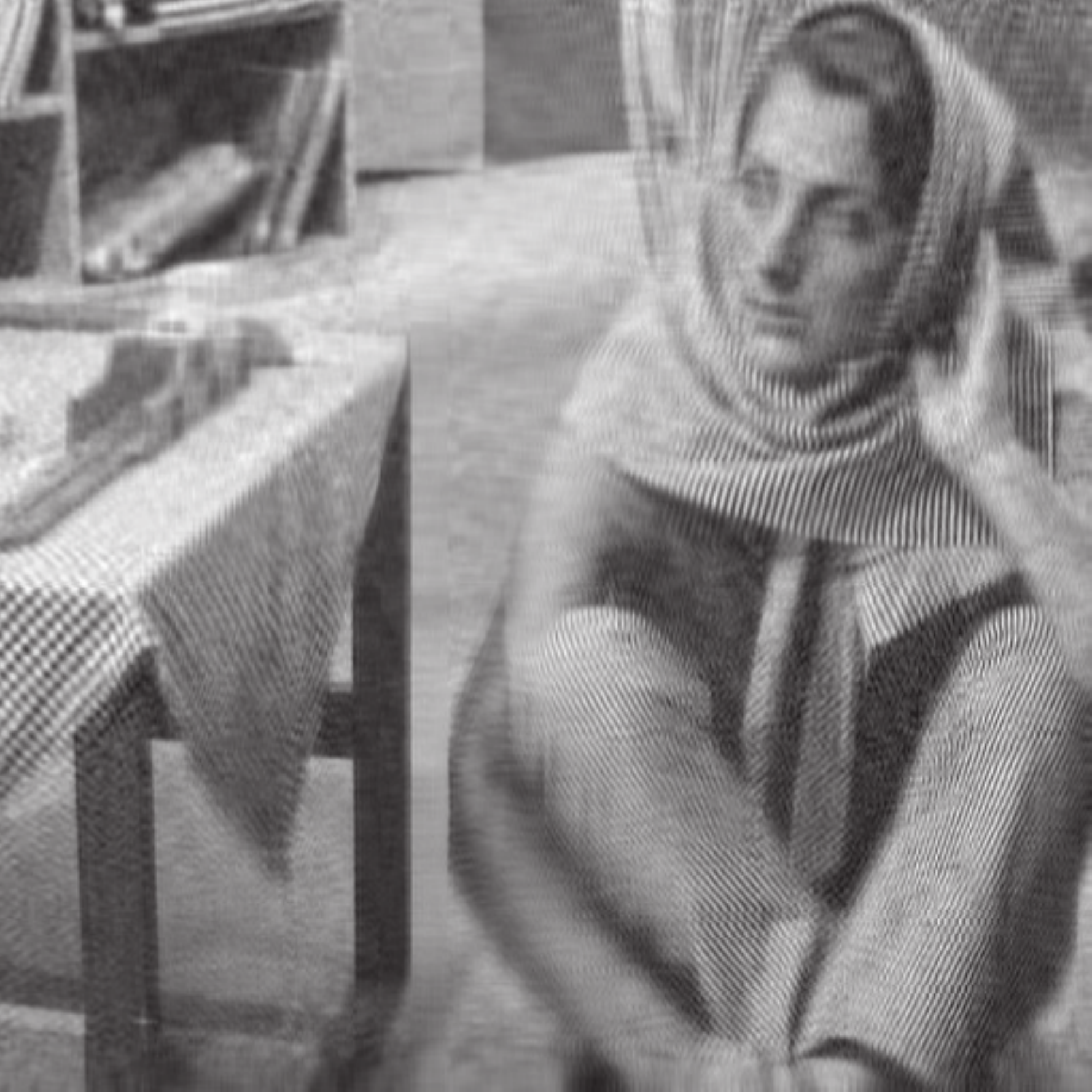} \includegraphics[width=0.32\linewidth]{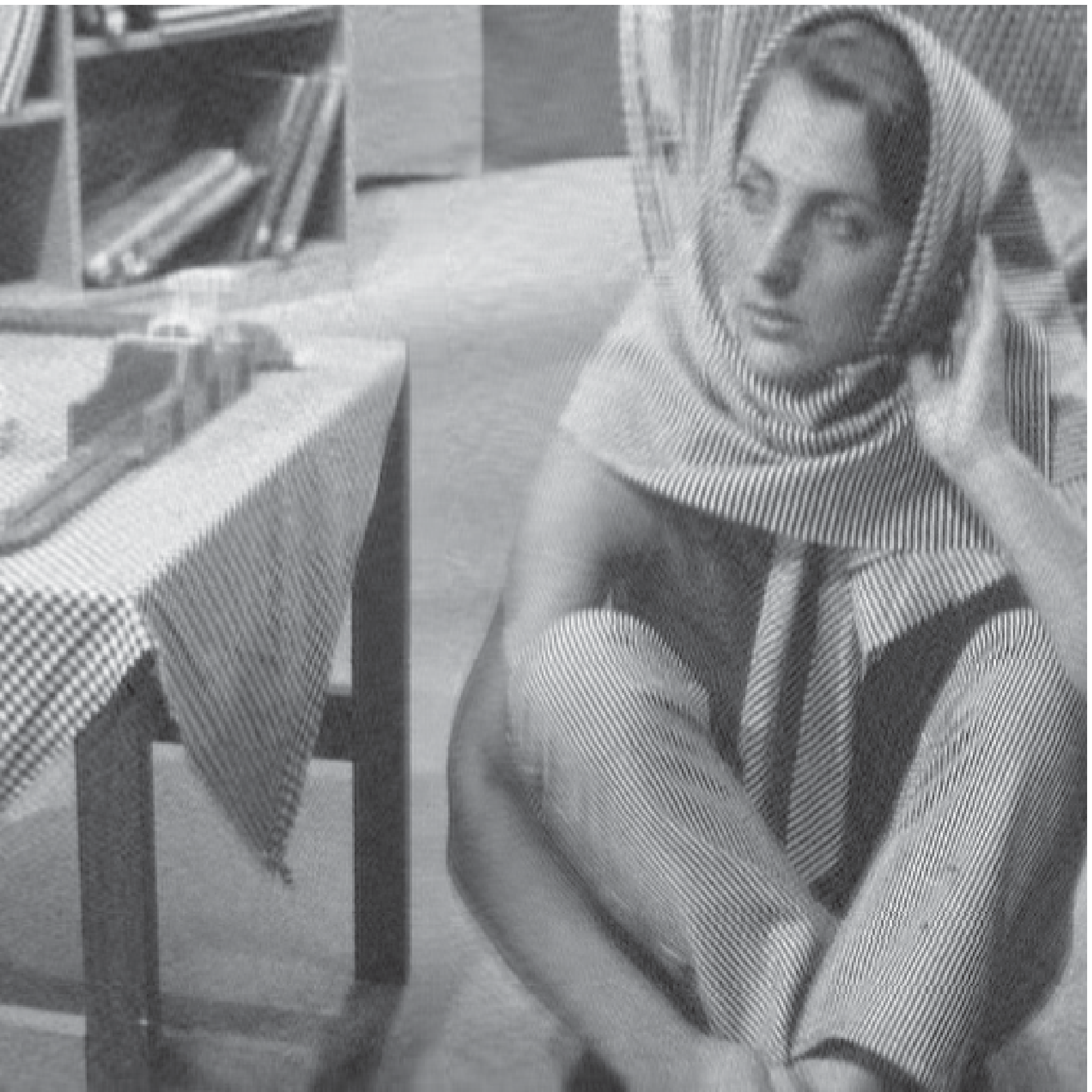}
\includegraphics[width=0.32\linewidth]{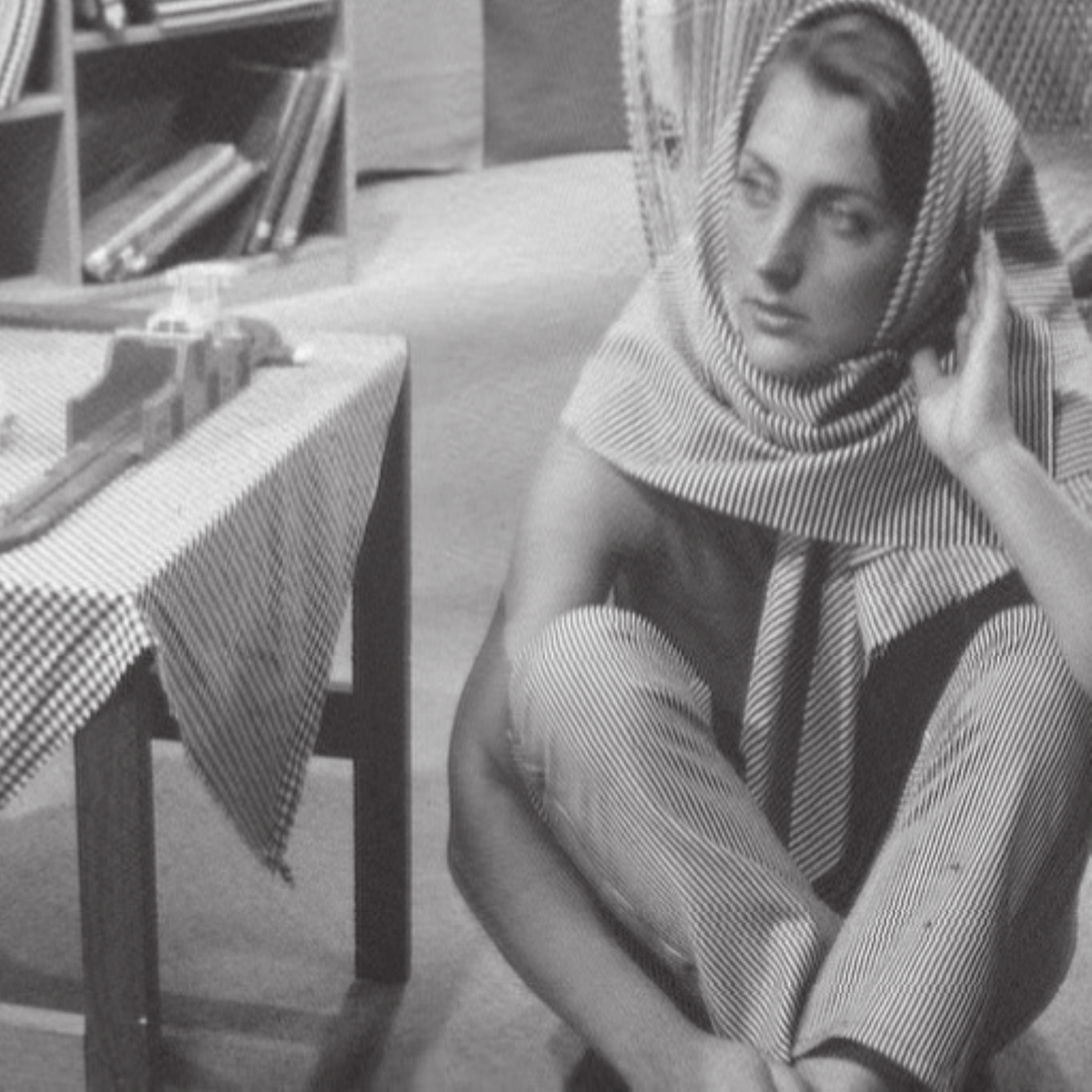}\\
 \includegraphics[width=0.32\linewidth]{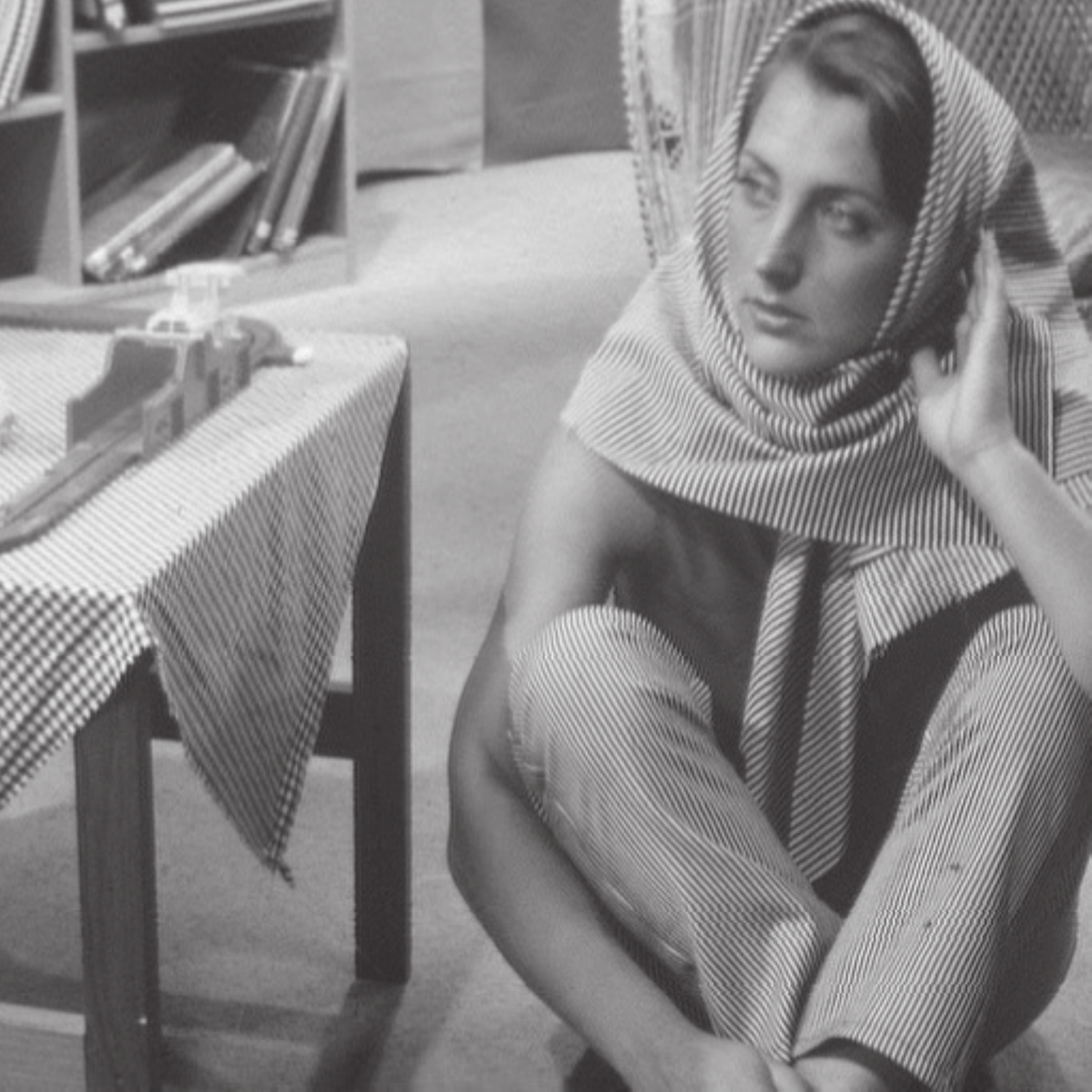} \includegraphics[width=0.32\linewidth]{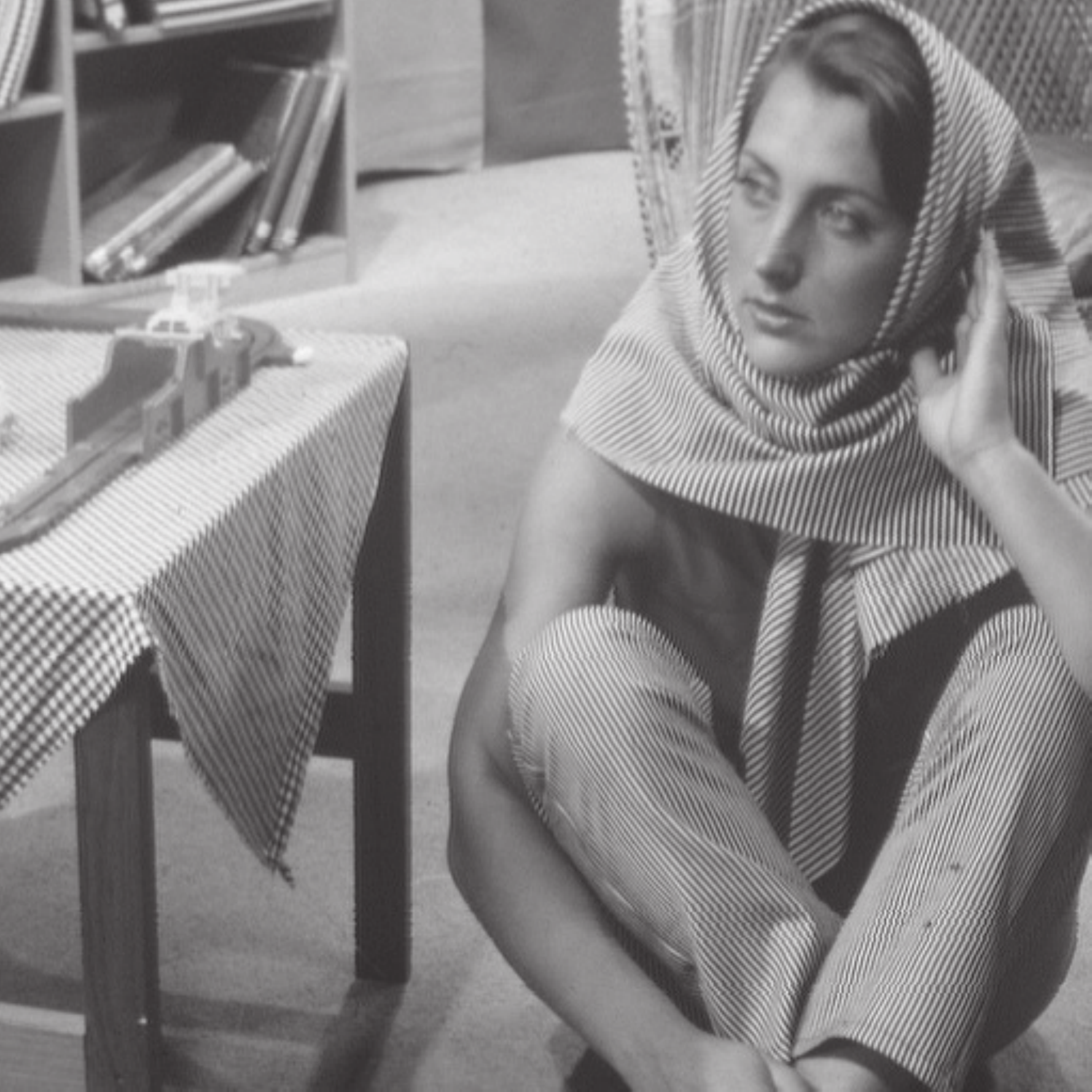}
\includegraphics[width=0.32\linewidth]{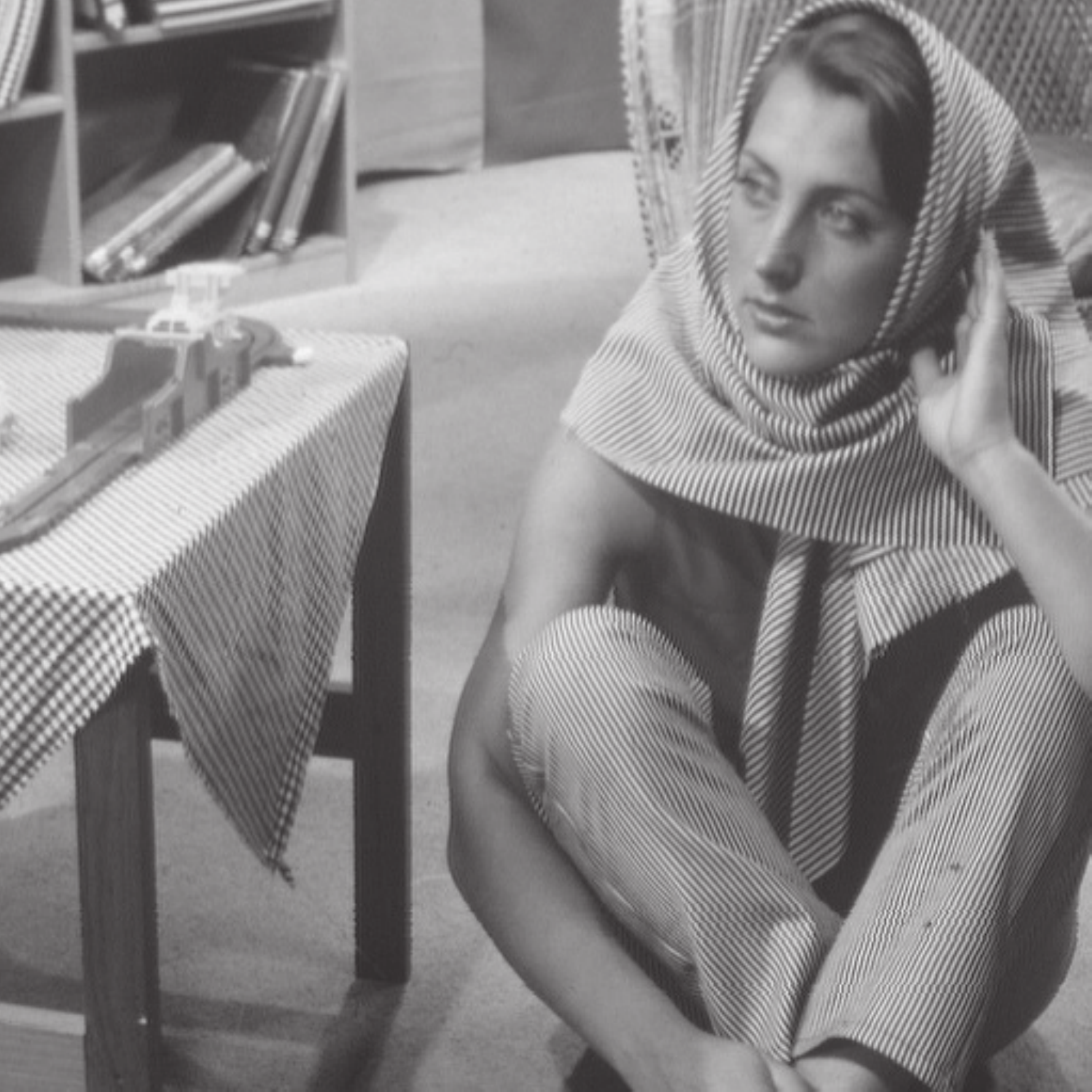}\\
 \includegraphics[width=0.32\linewidth]{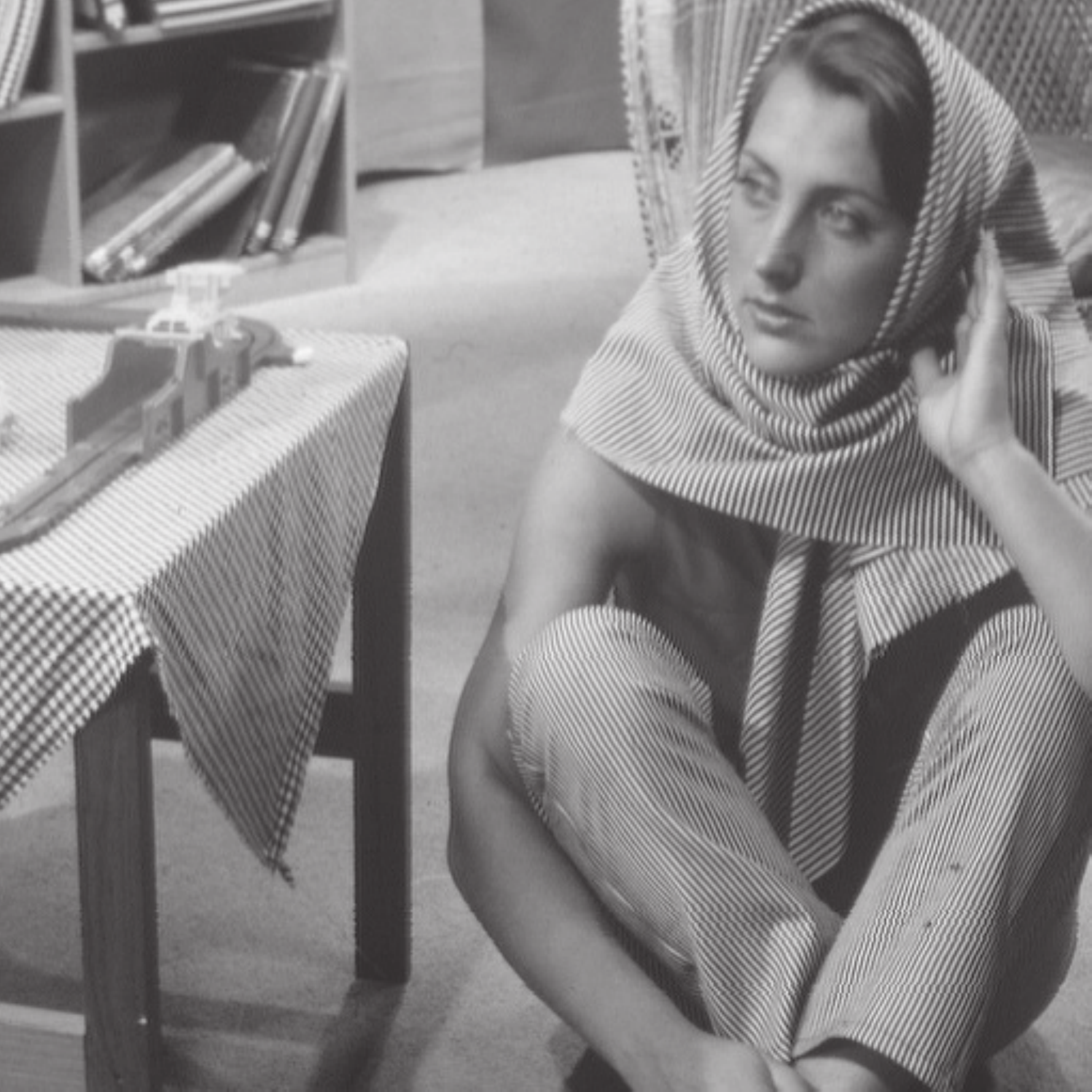} \includegraphics[width=0.32\linewidth]{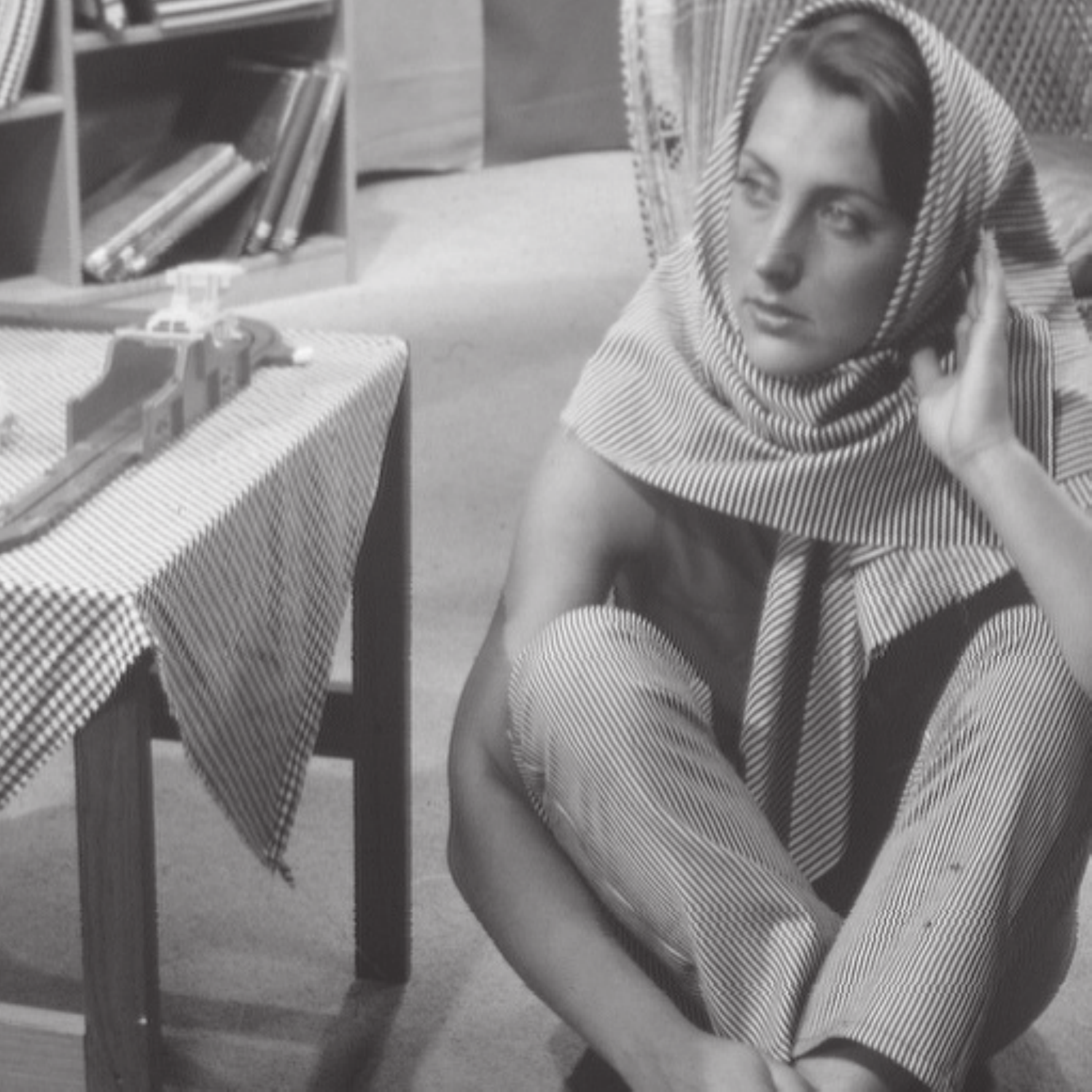}
\includegraphics[width=0.32\linewidth]{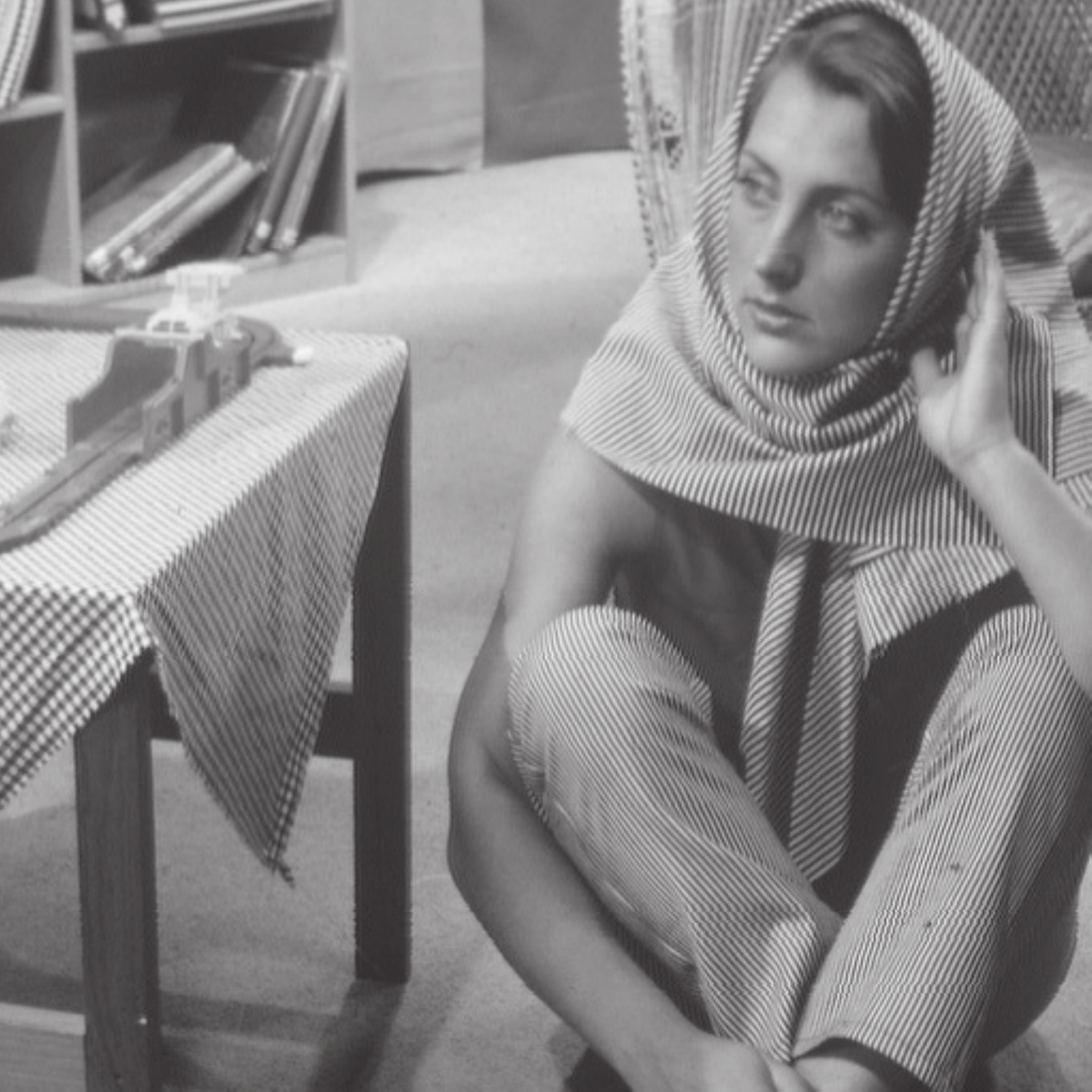}\\
 \includegraphics[width=0.32\linewidth]{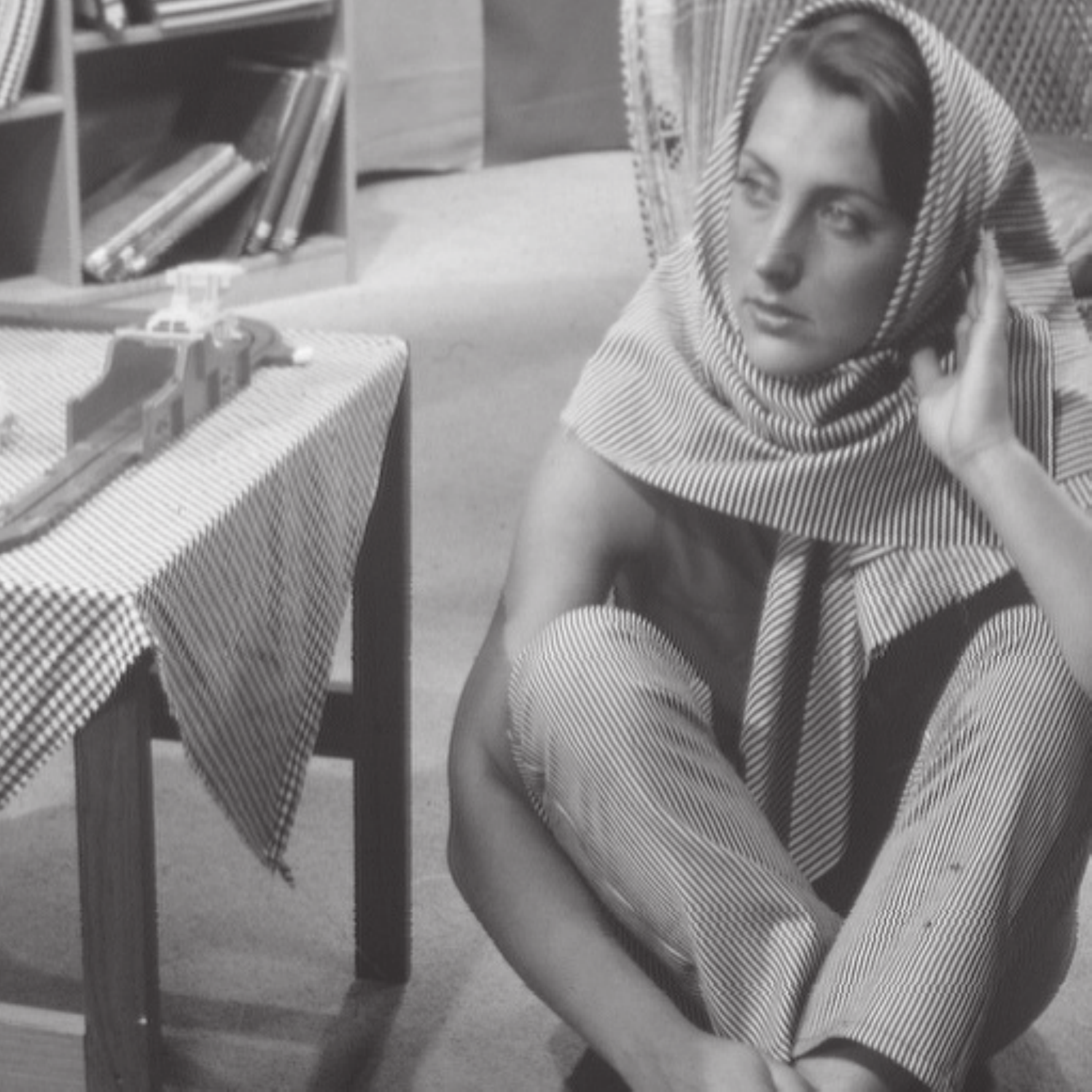} 
 \caption{Reconstructed images of Barbara using $10$, $20$, $30$, $40$,
$50$, $60$, $70$, $80$, $90$ and $100$ principal components
respectively from top to bottom, left to right.}
\end{figure}

\begin{center}
\begin{table}
\begin{centering}
\begin{tabular}{|cccc|}
\hline 
No. of P. C.s & CR & File Size & MSE\tabularnewline
{[}0.5ex{]} 10 & 10.0392 & 160.642 KB & 9.3646\tabularnewline
{[}0.5ex{]} 20 & 5.0196 & 172.250 KB & 5.5603\tabularnewline
{[}0.5ex{]} 30 & 3.3464 & 175.908 KB & 3.3438\tabularnewline
{[}0.5ex{]} 40 & 2.5098 & 176.492 KB & 1.9192\tabularnewline
{[}0.5ex{]} 50 & 2 & 176.436 KB & 1.0752\tabularnewline
{[}0.5ex{]} 60 & 1.6678 & 176.949 KB & 0.6133\tabularnewline
{[}0.5ex{]} 70 & 1.4302 & 178 KB & 0.3315\tabularnewline
{[}0.5ex{]} 80 & 1.2518 & 178 KB & 0.1477\tabularnewline
{[}0.5ex{]} 90 & 1.1130 & 178 KB & 0.0370\tabularnewline
{[}0.5ex{]} 100 & 1 & 178 KB & 1.0363e-10\tabularnewline
{[}0.5ex{]} &  &  & \tabularnewline
\hline 
\multicolumn{1}{c}{} &  &  & \multicolumn{1}{c}{}\tabularnewline
\end{tabular}
\par\end{centering}
\caption{\label{tab:table-name}Table for Barbara PCA compression using different
number of principal components used to reconstruct back the image.}
\end{table}
\par\end{center}

\section{\label{sec:kac}Frames, projections, and Kaczmarz algorithms}

In this section, we consider certain infinite products of projections.
Our framework is motivated by problems in approximation theory, in
harmonic analysis, in frame theory, and the context of the classical
Kaczmarz algorithm \cite{K-1937}. Traditionally, the infinite-dimensional
Kaczmarz algorithm is stated for sequences of vectors in a specified
Hilbert space $\mathscr{H}$, (typically, $\mathscr{H}$ is an $L^{2}$-space.)
We shall here formulate it instead for sequences of projections. As
a corollary, we get explicit and algorithmic criteria for convergence
of certain \emph{infinite products of projections} in $\mathscr{H}$.

\textbf{Motivation. }Our extension of the Kaczmarz algorithm to sequences
of projections is highly nontrivial: while in general convergence
questions for infinite products of projections (in Hilbert space)
is difficult (see e.g., \cite{MR0051437,MR647807,MR2129258,MR3796644}),
our projection-valued formulation of Kaczmarz' algorithm yields an
answer to this convergence question, as well as a number of applications
to stochastic analysis, and to frame-approximation questions in the
Hilbert space $L^{2}\left(\mu\right)$, where $\mu$ is in a class
of \emph{iterated function system} (IFS) measures (see \cite{MR625600,MR1656855,MR2319756,2016arXiv160308852H,MR3800275}). 

\emph{Literature guide}: In addition to Kaczmarz' pioneering paper
\cite{K-1937}, there are also the following more recent developments
of relevance to our present discussion \cite{MR1257097,MR1898684,MR2208766,MR2140451,MR2263965,MR2311862,MR2721177,MR2835851,MR3117886,MR3159297,MR3450541,MR3439812,MR3796634,MR3846956,MR3896982},
as well as \cite{2016arXiv160308852H,MR3796641,HERR2018}.\\

The classical Kaczmarz algorithm is an iterative method for solving
systems of linear equations, for example, $Ax=b$, where $A$ is an
$m\times n$ matrix.

Assume the system is consistent. Let $x_{0}$ be an arbitrary vector
in $\text{\ensuremath{\mathbb{R}^{n}}}$, and set 
\begin{equation}
x_{k}:=\mathop{argmin}_{\left\langle a_{j},x\right\rangle =b_{j}}\left\Vert x-x_{k-1}\right\Vert ^{2},\;k\in\mathbb{N};\label{eq:C1}
\end{equation}
where $j=k\mod m$, and $a_{j}$ denotes the $j^{th}$ row of $A$.
At each iteration, the minimizer is given by 
\begin{equation}
x_{k}=x_{k-1}+\frac{b_{j}-\left\langle a_{j},x_{k-1}\right\rangle }{\left\Vert a_{j}\right\Vert ^{2}}a_{j}.\label{eq:C2}
\end{equation}
That is, the algorithm recursively projects the current state onto
the hyperplane determined by the next row vector of $A$.

There is a stochastic version of (\ref{eq:C2}), where the row vectors
of $A$ are selected randomly \cite{MR2500924}. 

The Kaczmarz algorithm can be formulated in the Hilbert space setting
as follows:
\begin{defn}
Let $\left\{ e_{j}\right\} _{j\in\mathbb{N}_{0}}$ be a spanning set
of unit vectors in a Hilbert space $\mathscr{H}$, i.e., $span\left\{ e_{j}\right\} $
is dense in $\mathscr{H}$. For all $x\in\mathscr{H}$, let $x_{0}=e_{0}$,
and set 
\begin{equation}
x_{k}:=x_{k-1}+e_{k}\left\langle e_{k},x-x_{k-1}\right\rangle .\label{eq:C3}
\end{equation}
We say the sequence $\left\{ e_{j}\right\} _{j\in\mathbb{N}_{0}}$
is \emph{effective }if $\left\Vert x_{k}-x\right\Vert \rightarrow0$
as $k\rightarrow\infty$, for all $x\in\mathscr{H}$.
\end{defn}

\begin{rem}
A key motivation for our present analysis is an important result by
Stanis\l aw Kwapie\'{n}  and Jan Mycielski \cite{MR2263965}, giving
a criterion for stationary sequences (referring to a suitable $L^{2}\left(\mu\right)$)
to be effective.
\end{rem}

\textbf{Observation.} Equation (\ref{eq:C3}) yields, by forward induction:
\begin{eqnarray*}
x-x_{k} & = & \left(1-P_{k}\right)\left(x-x_{k-1}\right)\\
 & = & \left(1-P_{k}\right)\left(1-P_{k-1}\right)\left(x-x_{k-2}\right)\\
 & \vdots\\
 & = & \left(1-P_{k}\right)\left(1-P_{k-1}\right)\cdots\left(1-P_{0}\right)x,
\end{eqnarray*}
where $P_{j}$ is the orthogonal projection onto $e_{j}$. Throughout,
we shall use ``$1$'' also for the identity operator. This motivates
the following:
\begin{defn}
\label{def:efp}A system $\left\{ P_{j}\right\} _{j\in\mathbb{N}_{0}}$
of orthogonal projections in $\mathscr{H}$ is said to be \emph{effective}
if 
\begin{equation}
T_{n}\coloneqq\left(1-P_{n}\right)\left(1-P_{n-1}\right)\cdots\left(1-P_{0}\right)\xrightarrow{\;s\;}0,\label{eq:pc8}
\end{equation}
i.e., $T_{n}$ converges to zero as $n\rightarrow\infty$ in the strong
operator topology. 
\end{defn}

\begin{prop}
\label{prop:pc1}Let $\left\{ P_{j}\right\} _{j\in\mathbb{N}_{0}}$
be a sequence of orthogonal projections in a Hilbert space $\mathscr{H}$.
Suppose there exists $0<c<1$ such that 
\begin{equation}
\left\Vert P_{j}\left(1-P_{j-1}\right)y\right\Vert ^{2}\geq c\left\Vert \left(1-P_{j-1}\right)y\right\Vert ^{2}\label{eq:pc1}
\end{equation}
for all $y\in\mathscr{H}$, and all $j\in\mathbb{N}$. Then the system
$\left\{ P_{j}\right\} _{j\in\mathbb{N}_{0}}$ is effective.
\end{prop}

\begin{proof}
Assume (\ref{eq:pc1}) holds. Then, for all $x\in\mathscr{H}$, 
\begin{eqnarray*}
 &  & \left\Vert \left(1-P_{n}\right)\left(1-P_{n-1}\right)\cdots\left(1-P_{0}\right)x\right\Vert ^{2}\\
 & = & \left\Vert \left(1-P_{n-1}\right)\cdots\left(1-P_{0}\right)x\right\Vert ^{2}-\left\Vert P_{n}\left(1-P_{n-1}\right)\cdots\left(1-P_{0}\right)x\right\Vert ^{2}\\
 & \leq & \left\Vert \left(1-P_{n-1}\right)\cdots\left(1-P_{0}\right)x\right\Vert ^{2}-c\left\Vert \left(1-P_{n-1}\right)\cdots\left(1-P_{0}\right)x\right\Vert ^{2}\\
 & = & \left(1-c\right)\left\Vert \left(1-P_{n-1}\right)\cdots\left(1-P_{0}\right)x\right\Vert ^{2}\\
 & \leq & \left(1-c\right)^{2}\left\Vert \left(1-P_{n-2}\right)\cdots\left(1-P_{0}\right)x\right\Vert ^{2}\\
 & \vdots\\
 & \leq & \left(1-c\right)^{n}\left\Vert \left(1-P_{0}\right)x\right\Vert ^{2}\rightarrow0,\;\text{as \ensuremath{n\rightarrow\infty}}.
\end{eqnarray*}
\end{proof}
\begin{rem}
Condition (\ref{eq:pc1}) may be replaced by 
\[
\left\Vert P_{j}\left(1-P_{j-1}\right)y\right\Vert ^{2}\geq c_{j}\left\Vert \left(1-P_{j-1}\right)y\right\Vert ^{2},\quad\forall y\in\mathscr{H},
\]
with $0<c_{j}<1$. Then the proof of \propref{pc1} is modified as:
\begin{eqnarray*}
 &  & \left\Vert \left(1-P_{n}\right)\left(1-P_{n-1}\right)\cdots\left(1-P_{0}\right)x\right\Vert ^{2}\\
 & \leq & \left(1-c_{n}\right)\left\Vert \left(1-P_{n-1}\right)\cdots\left(1-P_{0}\right)x\right\Vert ^{2}\\
 & \vdots\\
 & \leq & \left(1-c_{n}\right)\cdots\left(1-c_{1}\right)\left\Vert \left(1-P_{0}\right)x\right\Vert ^{2}\rightarrow0,\;\text{as \ensuremath{n\rightarrow\infty}}.
\end{eqnarray*}
\end{rem}

\begin{thm}
\label{thm:pc3}Let $\left\{ P_{j}\right\} _{j\in\mathbb{N}_{0}}$
be a sequence of orthogonal projections in a Hilbert space $\mathscr{H}$.
Set 
\begin{align}
T_{n} & =\left(1-P_{n}\right)\left(1-P_{n-1}\right)\cdots\left(1-P_{0}\right),\label{eq:pc3}\\
Q_{n} & =P_{n}\left(1-P_{n-1}\right)\cdots\left(1-P_{0}\right),\label{eq:pc4}
\end{align}
where $Q_{0}=P_{0}$. 

For all $n\in\mathbb{N}$, we have 
\[
\left\Vert x\right\Vert ^{2}=\left\Vert T_{n}x\right\Vert ^{2}+\sum_{k=0}^{n}\left\Vert Q_{k}x\right\Vert ^{2},\quad x\in\mathscr{H}.
\]
Hence $T_{n}\xrightarrow{\;s\:}0$ if and only if 
\begin{equation}
I=\sum_{j\in\mathbb{N}_{0}}Q_{j}^{*}Q_{j}.\label{eq:pc5}
\end{equation}
More precisely, (\ref{eq:pc5}) means that, 
\begin{equation}
\left\langle x,y\right\rangle =\sum_{j\in\mathbb{N}_{0}}\left\langle Q_{j}x,Q_{j}y\right\rangle ,\quad x,y\in\mathscr{H}.\label{eq:pc6}
\end{equation}
In particular, 
\begin{equation}
\left\Vert x\right\Vert ^{2}=\sum_{j\in\mathbb{N}_{0}}\left\Vert Q_{j}x\right\Vert ^{2},\quad x\in\mathscr{H}.\label{eq:pc7}
\end{equation}
\end{thm}

\begin{proof}
Note that 
\begin{eqnarray*}
\left\Vert T_{n}x\right\Vert ^{2} & = & \left\Vert \left(1-P_{n}\right)\left(1-P_{n-1}\right)\cdots\left(1-P_{0}\right)x\right\Vert ^{2}\\
 & = & \left\Vert \left(1-P_{n-1}\right)\cdots\left(1-P_{0}\right)x\right\Vert ^{2}-\left\Vert P_{n}\left(1-P_{n-1}\right)\cdots\left(1-P_{0}\right)x\right\Vert ^{2}\\
 & = & \left\Vert T_{n-1}x\right\Vert ^{2}-\left\Vert Q_{n}x\right\Vert ^{2}\\
 & = & \left\Vert T_{n-2}x\right\Vert ^{2}-\left\Vert Q_{n-1}x\right\Vert ^{2}-\left\Vert Q_{n}x\right\Vert ^{2}\\
 & \vdots\\
 & = & \left\Vert \left(1-P_{0}\right)x\right\Vert ^{2}-\left\Vert Q_{1}x\right\Vert ^{2}-\cdots-\left\Vert Q_{n-1}x\right\Vert ^{2}-\left\Vert Q_{n}x\right\Vert ^{2}\\
 & = & \left\Vert x\right\Vert ^{2}-\left\Vert Q_{0}x\right\Vert ^{2}-\left\Vert Q_{1}x\right\Vert ^{2}-\cdots-\left\Vert Q_{n-1}x\right\Vert ^{2}-\left\Vert Q_{n}x\right\Vert ^{2}.
\end{eqnarray*}
Therefore 
\[
T_{n}\xrightarrow{\;s\:}0\Longleftrightarrow\left\Vert x\right\Vert ^{2}=\sum_{j\in\mathbb{N}_{0}}\left\Vert Q_{j}x\right\Vert ^{2}.
\]
\end{proof}
\begin{rem}
The system of operators $\left\{ Q_{j}\right\} _{j\in\mathbb{N}_{0}}$
in \thmref{pc3} has frame-like properties, see (\ref{eq:pc5})--(\ref{eq:pc7}).
Specifically, the mapping 
\[
\mathscr{H}\ni x\xmapsto{\;V\;}\left(Q_{j}x\right)\in l^{2}\left(\mathbb{N}_{0}\right)\otimes\mathscr{H}
\]
plays the role of an analysis operator, and the synthesis operator
$V^{*}$ is given by 
\[
l^{2}\left(\mathbb{N}_{0}\right)\otimes\mathscr{H}\ni\xi\xmapsto{\;V^{*}\;}\sum_{j\in\mathbb{N}_{0}}Q_{j}^{*}\xi_{j}.
\]
Note that $1=V^{*}V$, and (\ref{eq:pc7}) is the \emph{generalized
Parseval identity}. 
\end{rem}

\subsection{The case of rank-1 projections in $\mathscr{H}$}

Let $\left\{ P_{j}\right\} _{j\in\mathbb{N}_{0}}$ be a system of
rank-1  projections, i.e., $P_{j}=\left|e_{j}\left\rangle \right\langle e_{j}\right|$,
where $\left\{ e_{j}\right\} _{j\in\mathbb{N}_{0}}$ is a set of unit
vectors in $\mathscr{H}$. When the system $\left\{ e_{j}\right\} $
is independent, then the corresponding family of projections $P_{j}=\left|e_{j}\left\rangle \right\langle e_{j}\right|$
is non-commutative.
\begin{cor}
Suppose all the $P_{j}$'s are of rank-1, i.e., $P_{j}=\left|e_{j}\left\rangle \right\langle e_{j}\right|$
where $\left\{ e_{j}\right\} _{j\in\mathbb{N}_{0}}$ is a set of unit
vectors in $\mathscr{H}$. Assume $\left\{ P_{j}\right\} _{j\in\mathbb{N}_{0}}$
is effective (\defref{efp}). Set 
\begin{align*}
g_{0} & =e_{0},\\
g_{n} & =\left(1-P_{0}\right)\left(1-P_{1}\right)\cdots\left(1-P_{n-1}\right)e_{n},\quad n\in\mathbb{N}.
\end{align*}
Then $\left\{ g_{n}\right\} _{n\in\mathbb{N}_{0}}$ is a Parseval
frame in $\mathscr{H}$. 

Specifically, we have 
\begin{align}
g_{0} & =e_{0}\nonumber \\
g_{n} & =e_{n}-\sum_{j=0}^{n-1}\left\langle e_{j},e_{n}\right\rangle g_{j},\quad n\in\mathbb{N}.\label{eq:gn}
\end{align}
\end{cor}

\begin{proof}
Recall that 
\[
Q_{n}=P_{n}\left(1-P_{n-1}\right)\cdots\left(1-P_{0}\right),
\]
see (\ref{eq:pc4}). 

Since $P_{n}$ has rank-1, it follows that $Q_{n}$ has the form 
\[
Q_{n}=\left|e_{n}\left\rangle \right\langle g_{n}\right|
\]
for some $g_{n}\in\mathscr{H}$, which in turn is given by 
\begin{eqnarray*}
g_{n} & = & Q_{n}^{*}e_{n}=\left(1-P_{0}\right)\left(1-P_{1}\right)\cdots\left(1-P_{n-1}\right)P_{n}e_{n}\\
 & = & \left(1-P_{0}\right)\left(1-P_{1}\right)\cdots\left(1-P_{n-1}\right)e_{n}\\
 & = & \left(1-P_{0}\right)\left(1-P_{1}\right)\cdots\left(1-P_{n-2}\right)\left(e_{n}-\left\langle e_{n-1},e_{n}\right\rangle e_{n-1}\right)\\
 & = & \left(1-P_{0}\right)\left(1-P_{1}\right)\cdots\left(1-P_{n-2}\right)e_{n}-\left\langle e_{n-1},e_{n}\right\rangle g_{n-1}\\
 & = & \left(1-P_{0}\right)\left(1-P_{1}\right)\cdots\left(1-P_{n-3}\right)e_{n}-\left\langle e_{n-2},e_{n}\right\rangle g_{n-2}-\left\langle e_{n-1},e_{n}\right\rangle g_{n-1}\\
 & \vdots\\
 & = & e_{n}-\sum_{j=0}^{n-1}\left\langle e_{j},e_{n}\right\rangle g_{j}.
\end{eqnarray*}

Now, 
\[
\left\Vert x\right\Vert ^{2}=\sum_{j\in\mathbb{N}_{0}}\left\Vert Q_{j}x\right\Vert ^{2}=\sum_{j\in\mathbb{N}_{0}}\left|\left\langle g_{j},x\right\rangle \right|^{2},\quad x\in\mathscr{H}.
\]
\end{proof}
\begin{cor}
The system $\left\{ \left|e_{j}\left\rangle \right\langle e_{j}\right|\right\} _{j\in\mathbb{N}_{0}}$
is effective iff $\left\{ g_{j}\right\} _{j\in\mathbb{N}_{0}}$ (see
(\ref{eq:gn})) is a Parseval frame in $\mathscr{H}$.
\end{cor}

\begin{example}
Consider a positive definite function $G\times G\xrightarrow{\;K\;}\mathbb{C}$
on a lca group $G$, where $K\left(x,y\right)=K\left(x-y\right)$
and $K\left(0\right)=1$. Let $\mathscr{H}_{K}$ be the associated
RKHS. 

Fix a discrete subset $\left\{ x_{j}\right\} _{j\in\mathbb{N}_{0}}\subset G$,
and define 
\begin{equation}
P_{j}=1-\left|K_{x_{j}}\left\rangle \right\langle K_{x_{j}}\right|,\quad j\in\mathbb{N}_{0}.\label{eq:pc9}
\end{equation}
\end{example}

\begin{cor}
Assume there exists $0<c<1$ such that 
\begin{equation}
K\left(x_{j}-x_{j-1}\right)^{2}\leq1-c,\quad j\in\mathbb{N}.\label{eq:pc10}
\end{equation}
Then the system $\left\{ P_{j}\right\} _{j\in\mathbb{N}_{0}}$ in
(\ref{eq:pc9}) is effective. 

For the operator valued frame $\left\{ Q_{j}\right\} _{j\in\mathbb{N}_{0}}$
in \thmref{pc3}, it holds that 
\[
rank\left(Q_{j}\right)\leq2,\quad j\in\mathbb{N}_{0}.
\]
\end{cor}

\begin{proof}
In the current setting, condition (\ref{eq:pc1}) in \propref{pc1}
translates to 
\begin{gather*}
\left\Vert \left(1-\left|K_{x_{j}}\left\rangle \right\langle K_{x_{j}}\right|\right)K_{x_{j-1}}\right\Vert _{\mathscr{H}_{K}}^{2}\geq c\left\Vert K_{x_{j-1}}\right\Vert _{\mathscr{H}_{K}}^{2}=c\\
\Updownarrow\\
\left\Vert K_{x_{j-1}}-K\left(x_{j}-x_{j-1}\right)K_{x_{j}}\right\Vert _{\mathscr{H}_{K}}^{2}\geq c\\
\Updownarrow\\
K\left(x_{j}-x_{j-1}\right)^{2}\leq1-c.
\end{gather*}
Therefore $\left\{ P_{j}\right\} _{j\in\mathbb{N}_{0}}$ is effective,
provided that (\ref{eq:pc10}) is satisfied. 

Let $Q_{j}$ be as in \thmref{pc3}. That is, for all $f\in\mathscr{H}_{K}$
(see \corref{psin}), 
\begin{eqnarray*}
Q_{j}f & = & P_{j}\left(1-P_{j-1}\right)\cdots\left(1-P_{0}\right)f\\
 & = & \left(1-\left|K_{x_{j}}\left\rangle \right\langle K_{x_{j}}\right|\right)\left(\left|K_{x_{j-1}}\left\rangle \right\langle K_{x_{j-1}}\right|\right)\cdots\left(\left|K_{x_{0}}\left\rangle \right\langle K_{x_{0}}\right|\right)f\\
 & = & K_{x_{j-1}}K\left(x_{j-1}-x_{j-2}\right)\cdots K\left(x_{1}-x_{0}\right)f\left(x_{0}\right)\\
 & \quad & -K_{x_{j}}K\left(x_{k_{j}}-x_{k_{j-1}}\right)K\left(x_{j-1}-x_{j-2}\right)\cdots K\left(x_{1}-x_{0}\right)f\left(x_{0}\right).
\end{eqnarray*}

In particular, 
\[
Q_{j}f\in span\left\{ K_{x_{j-1}},K_{x_{j}}\right\} 
\]
and so $rank\left(Q_{j}\right)\leq2$. 
\end{proof}

\section{\label{sec:pw}Paley Wiener spaces}
\begin{defn}[Paley Wiener spaces (see, e.g., \cite{MR3711878})]
 Let $\Omega$ be a bounded subset of $\mathbb{R}^{d}$, and let
$L^{2}=L^{2}\left(\lambda_{d}\right)=L^{2}\left(\mathbb{R}^{d},\lambda_{d}\right)$
with $\lambda_{d}=$ the usual Lebesgue measure on $\mathbb{R}^{d}$.
Set 
\[
\text{PW}\left(\Omega\right)\coloneqq\left\{ f\in L^{2}:supp(\widehat{f})\subseteq\Omega\right\} 
\]
where 
\begin{equation}
\widehat{f}\left(\xi\right)=\int_{\mathbb{R}^{d}}e^{-i\xi\cdot x}f\left(x\right)dx\label{eq:pw1}
\end{equation}
is the Fourier transform, and $x=\left(x_{1},\cdots,x_{d}\right)\in\mathbb{R}^{d}$,
$dx=\lambda_{d}\left(x\right)$. 
\end{defn}

\begin{lem}
$\text{PW}\left(\Omega\right)$ is a RKHS. 
\end{lem}

\begin{proof}
Let $f\in\text{PW}\left(\Omega\right)$. Since $\Omega$ is bounded
and $\lambda_{d}\left(\Omega\right)<\infty$, then 
\begin{equation}
f\left(x\right)=\frac{1}{\left(2\pi\right)^{d}}\int_{\Omega}e^{i\xi\cdot x}\widehat{f}\left(\xi\right)d\xi\label{eq:pw2}
\end{equation}
where $\xi=\left(\xi_{1},\cdots,\xi_{d}\right)\in\mathbb{R}^{d}$,
and $d\xi=\lambda_{d}\left(\xi\right)$.

For $f\in\text{PW}\left(\Omega\right)$, we have 
\begin{align}
\left\Vert f\right\Vert _{\text{PW}\left(\Omega\right)}^{2} & \coloneqq\left\Vert f\right\Vert _{L^{2}\left(\mathbb{R}^{d}\right)}^{2}=\int_{\mathbb{R}^{d}}\left|f\left(x\right)\right|^{2}dx\nonumber \\
 & =\int_{\mathbb{R}^{d}}\big|\widehat{f}\left(\xi\right)\big|^{2}d\xi=\int_{\Omega}\big|\widehat{f}\big|^{2}d\lambda<\infty.\label{eq:pw3}
\end{align}
From (\ref{eq:pw2}), 
\[
\left|f\left(x\right)\right|\leq\frac{1}{\left(2\pi\right)^{d}}\lambda_{d}\left(\Omega\right)\left(\int_{\Omega}\big|\widehat{f}\big|^{2}d\lambda\right)^{1/2}=C_{d}\left\Vert f\right\Vert _{\text{PW}\left(\Omega\right)}
\]
where $c_{d}\coloneqq\frac{\lambda_{d}\left(\Omega\right)}{\left(2\pi\right)^{d}}$
. That is, point evaluation at every $x\in\mathbb{R}^{d}$ is a bounded
linear functional on $\text{PW}\left(\Omega\right)$. 

Moreover, the reproducing kernel is given by 
\[
K_{\text{PW}\left(\Omega\right)}\left(x,y\right)\coloneqq\widehat{\chi}_{\Omega}\left(x-y\right),\quad x,y\in\mathbb{R}^{d},
\]
where $\widehat{\chi}_{\Omega}$ denotes the Fourier transform of
the indicator function 
\[
\chi_{\Omega}\left(\xi\right)=\begin{cases}
1 & \xi\in\Omega\\
0 & \xi\in\mathbb{R}^{d}\backslash\Omega.
\end{cases}
\]
\end{proof}
Proposed duality:
\begin{equation}
\xymatrix{\boxed{\text{\ensuremath{\begin{matrix}\text{countably discrete subset \ensuremath{V\subset\mathbb{R}^{d}} with}\\
\text{generalized interpolation properties}
\end{matrix}}}}\ar@{<=>}[rr] &  & \boxed{\text{\ensuremath{\begin{matrix}\Omega\subset\mathbb{R}^{d}\\
\lambda_{d}\left(\Omega\right)<\infty
\end{matrix}}}}}
\label{eq:P5}
\end{equation}
When $\Omega$ and $V$ are as above, see (\ref{eq:P5}), consider
the question of existence of finite constants $0<A\leq B<\infty$
such that 
\begin{equation}
A\left\Vert f\right\Vert _{\text{PW}\left(\Omega\right)}^{2}\leq\sum_{x\in V}\left|f\left(x\right)\right|^{2}\leq B\left\Vert f\right\Vert _{\text{PW}\left(\Omega\right)}^{2},\quad\forall f\in\text{PW}\left(\Omega\right);\label{eq:pw6}
\end{equation}
or with a weight function $w:V\rightarrow\mathbb{R}_{+}$, 
\begin{equation}
A\left\Vert f\right\Vert _{\text{PW}\left(\Omega\right)}^{2}\leq\sum_{x\in V}w\left(x\right)\left|f\left(x\right)\right|^{2}\leq B\left\Vert f\right\Vert _{\text{PW}\left(\Omega\right)}^{2},\quad\forall f\in\text{PW}\left(\Omega\right).\label{eq:pw7}
\end{equation}

\begin{defn}
Give $\Omega\subset\mathbb{R}^{d}$, $\lambda_{d}\left(\Omega\right)<\infty$,
set 
\begin{align*}
\text{SAMP}\left(\Omega\right) & \coloneqq\big\{ V:V\subset\mathbb{R}^{d}\:\text{countably discrete, and \ensuremath{\left(\ref{eq:pw6}\right)} or \ensuremath{\left(\ref{eq:pw7}\right)}}\\
 & \quad\quad\text{holds for some finite constants \ensuremath{A,B,} \ensuremath{0<A\leq B<\infty}}\big\};\\
\text{FREBD}\left(V\right) & \coloneqq\big\{\Omega:\lambda_{d}\left(\Omega\right)<\infty,\:\text{and \ensuremath{\left(\ref{eq:pw6}\right)} or \ensuremath{\left(\ref{eq:pw7}\right)} holds for some }\\
 & \quad\quad\text{constants \ensuremath{A,B,} \ensuremath{0<A\leq B<\infty}}\big\}.
\end{align*}
\end{defn}

\begin{problem*}
(i) Given $\Omega$, find $\text{SAMP}\left(\Omega\right)$; (ii)
Given $V$, find $\text{FREBD}\left(V\right)$. 
\end{problem*}

\section{\label{sec:gp}Gaussian Processes and Gaussian Hilbert Spaces}

By a probability space, we mean a triple $\left(\Omega,\mathscr{F},\mathbb{P}\right)$
where 
\begin{itemize}
\item $\Omega$: set of sample points,
\item $\mathscr{F}$: $\sigma$-algebra of events (subsets of $\Omega$),
\item $\mathbb{P}$: a probability measure defined on $\mathscr{F}$. 
\end{itemize}
A random variable 
\begin{equation}
K:\Omega\rightarrow\mathbb{R}\:\left(\text{\ensuremath{\mathbb{C}}, or a Hilbert space}\right)
\end{equation}
is a measurable function defined on $\left(\Omega,\mathscr{F}\right)$,
i.e., we require that for Borel sets $B$ (in $\mathbb{R}$, or $\mathbb{C}$),
and cylinder sets (referring to a fixed Hilbert space $\mathscr{H}$)
we have $K^{-1}\left(B\right)\in\mathscr{F}$ where 
\begin{equation}
K^{-1}\left(B\right)=\left\{ \omega\in\Omega:K\left(\omega\right)\in B\right\} .
\end{equation}
The distribution $\mu_{K}$ of $K$ is the measure 
\begin{equation}
\mu_{K}\coloneqq\mathbb{P}\circ K^{-1}.
\end{equation}
If $\mu_{K}$ is Gaussian, we say that $K$ is a Gaussian random variable.
A Gaussian process is a system $\left\{ K_{x}:x\in X\right\} $ of
random variables (refer to $\left(\Omega,\mathscr{F},\mathbb{P}\right)$),
indexed by some set $X$, in this case $X$. 

Here we shall restrict to the case of a Gaussian process, and we shall
assume 
\begin{equation}
\mathbb{E}\left(K_{x}\right)=0,\quad\forall x\in X;
\end{equation}
where 
\begin{equation}
\mathbb{E}\left(\cdot\right)=\int_{\Omega}\left(\cdot\right)d\mathbb{P}
\end{equation}
denotes expectation w.r.t. $\mathbb{P}$.

If $\mu_{K}\in N\left(0,1\right)$, i.e., 
\begin{equation}
\mu_{K}\left(t\right)=\frac{1}{\sqrt{2\pi}}e^{-t^{2}/2},\quad t\in\mathbb{R},
\end{equation}
we say that $K$ (or $\mu_{K}$) is a standard Gaussian. 
\begin{lem}
Let $\left\{ Z_{n}\right\} _{n\in\mathbb{N}_{0}}$ be a system of
independent identically distributed $N\left(0,1\right)$s (i.i.d $N\left(0,1\right)$).
Let $\mathscr{H}$ be a Hilbert space, and $\left\{ Q_{n}\right\} _{n\in\mathbb{N}_{0}}$
a system of projections as in \thmref{pc3}, i.e., 
\begin{equation}
\sum_{n\in\mathbb{N}_{0}}\left\langle Q_{n}u,Q_{n}v\right\rangle _{\mathscr{H}}=\left\langle u,v\right\rangle _{\mathscr{H}},\quad\forall u,v\in\mathscr{H}.\label{eq:G7}
\end{equation}
Then 
\begin{equation}
W\left(\cdot\right)=W^{\left(Q,\mathscr{H}\right)}\left(\cdot\right)\coloneqq\sum_{n\in\mathbb{N}_{0}}Q_{n}Z_{n}\left(\cdot\right)\label{eq:G8}
\end{equation}
 defines an operator valued Gaussian process, and 
\begin{equation}
\mathbb{E}\left(\left\langle W\left(\cdot\right)u,W\left(\cdot\right)v\right\rangle _{\mathscr{H}}\right)=\left\langle u,v\right\rangle ,\quad\forall u,v\in\mathscr{H}.\label{eq:G9}
\end{equation}
\end{lem}

\begin{proof}[Proof sketch]
Fix $u,v\in\mathscr{H}$; then 
\begin{eqnarray*}
\text{LHS}_{\left(\ref{eq:G9}\right)} & \underset{\text{by \ensuremath{\left(\ref{eq:G8}\right)}}}{=} & \underset{\mathbb{N}_{0}\times\mathbb{N}_{0}}{\sum\sum}\left\langle Q_{n}u,Q_{m}v\right\rangle _{\mathscr{H}}\underset{=\delta_{n,m}}{\underbrace{\mathbb{E}\left(Z_{n}Z_{m}\right)}}\\
 & = & \sum_{n\in\mathbb{N}_{0}}\left\langle Q_{n}u,Q_{n}v\right\rangle _{\mathscr{H}}\\
 & \underset{\text{by \ensuremath{\left(\ref{eq:G7}\right)}}}{=} & \left\langle u,v\right\rangle _{\mathscr{H}}.
\end{eqnarray*}
\end{proof}
\begin{cor}
Let $\left\{ Q_{n}\right\} _{n\in\mathbb{N}_{0}}$ be an effective
system in $\mathscr{H}_{K}$, where $K:X\times X\rightarrow\mathbb{R}$
is a given p.d. kernel and $\mathscr{H}_{K}$ the associated RKHs.
Then $W$ from (\ref{eq:G8}) has the property that 
\[
K\left(x,y\right)=\mathbb{E}\left(\left\langle W\left(\cdot\right)K_{x},W\left(\cdot\right)K_{y}\right\rangle _{\mathscr{H}_{K}}\right).
\]
\end{cor}

We recall the following theorem of Kolmogorov. It states that there
is a 1-1 correspondence between p.d. kernels on a set and mean zero
Gaussian processes indexed by the set. One direction is easy, and
the other is the deep part:
\begin{thm}[Kolmogorov]
\label{thm:kol} Let $X$ be a set. A function $K:X\times X\rightarrow\mathbb{C}$
is positive definite if and only if there is a Gaussian process $\left\{ W_{x}\right\} _{x\in X}$
realized in $L^{2}\left(\Omega,\mathscr{F},\mathbb{P}\right)$ with
mean zero, such that 
\begin{equation}
K\left(x,y\right)=\mathbb{E}\left[\overline{W}_{x}W_{y}\right].\label{eq:K1}
\end{equation}
 
\end{thm}

\begin{proof}
We refer to \cite{PaSc75} for the non-trivial direction. To stress
the idea, we include a proof of the easy part of the theorem: Assume
(\ref{eq:K1}). Let $\left\{ c_{i}\right\} _{i=1}^{n}\subset\mathbb{C}$
and $\left\{ x_{i}\right\} _{i=1}^{n}\subset X$, then we have
\[
\sum\nolimits _{i}\sum\nolimits _{j}\overline{c_{i}}c_{j}K\left(x_{i},x_{j}\right)=\mathbb{E}\left[\big|\sum c_{i}W_{x_{i}}\big|^{2}\right]\geq0,
\]
i.e., $K$ is p.d. 
\end{proof}
Let $\left(X,\mathscr{B},\nu\right)$ be a $\sigma$-finite measure
space, and let $\mathscr{B}_{fin}=\left\{ E\in\mathscr{B}:\nu\left(E\right)<\infty\right\} $.
Below we consider the following kernel $K$ on $\mathscr{B}_{fin}\times\mathscr{B}_{fin}$:
Set 
\begin{equation}
K\left(A,B\right)=\nu\left(A\cap B\right),\quad A,B\in\mathscr{B}_{fin}\label{eq:n4}
\end{equation}
and let $\mathscr{H}_{K^{\left(\nu\right)}}$ denote the associated
RKHS. 
\begin{prop}
~
\begin{enumerate}
\item $K=K^{\left(\nu\right)}$ in (\ref{eq:n4}) is positive definite.
\item $K^{\left(\nu\right)}$ is the covariance kernel for the stationary
Wiener process $W=W^{\left(\nu\right)}$ indexed by $\mathscr{B}_{fin}$,
i.e., Gaussian, mean zero, and 
\begin{equation}
\mathbb{E}\left(W_{A}W_{B}\right)=K^{\left(\nu\right)}\left(A,B\right)=\nu\left(A\cap B\right).\label{eq:n5}
\end{equation}
\item If $f\in L^{2}\left(\nu\right)$, and $W_{f}=\int_{X}f\left(x\right)dW_{x}$
denotes the corresponding Ito-integral, then 
\[
\mathbb{E}\left(\left|W_{f}\right|^{2}\right)=\int_{X}\left|f\right|^{2}d\nu;
\]
in particular, if $f=\sum_{i}\alpha_{i}\chi_{A_{i}}$, then 
\[
\sum\nolimits _{i}\sum\nolimits _{j}\alpha_{i}\alpha_{j}K^{\left(\nu\right)}\left(A_{i},A_{j}\right)=\int_{X}\left|\sum\nolimits _{i}\alpha_{i}\chi_{A_{i}}\right|^{2}d\nu.
\]
\item The RKHS $\mathscr{H}_{K^{\left(\nu\right)}}$ of the positive definite
kernel in (\ref{eq:n4}) consists of functions $F$ on $\mathscr{B}_{fin}$
represented by $f\in L^{2}\left(\nu\right)$ via 
\begin{equation}
F\left(A\right)=F_{f}\left(A\right)=\int_{A}fd\nu,\quad A\in\mathscr{B}_{fin};\label{eq:n5b}
\end{equation}
and 
\begin{equation}
\left\Vert F_{f}\right\Vert _{\mathscr{H}_{K^{\left(\nu\right)}}}^{2}=\left\Vert f\right\Vert _{L^{2}\left(\nu\right)}^{2}=\int_{X}\left|f\right|^{2}d\nu.\label{eq:n6}
\end{equation}
\item The map specified by 
\begin{equation}
\Psi\left(K^{\left(\nu\right)}\left(\cdot,A\right)\right)=\Psi\left(\nu\left(\left(\cdot\right)\cap A\right)\right)=\chi_{A},\quad\forall A\in\mathscr{B}_{fin}\label{eq:7}
\end{equation}
extends by linearity and by limits to an isometry 
\begin{equation}
\Psi:\mathscr{H}_{K^{\left(\nu\right)}}\longrightarrow L^{2}\left(\nu\right).\label{eq:q8}
\end{equation}
More generally if $F_{f}\in\mathscr{H}\left(K^{\left(\nu\right)}\right)$
is as in (\ref{eq:n5b}), then $\Psi\left(F_{f}\right)=f\in L^{2}\left(\nu\right)$. 
\end{enumerate}
\end{prop}

\begin{proof}
The details can be found at various places in the literature; see
e.g., \cite{MR2966130,MR3800275,MR3888850,MR4020693}. 
\end{proof}
\begin{prop}
Let $\mathscr{B}_{fin}\times\mathscr{B}_{fin}\xrightarrow{\;K\;}\mathbb{R}$
be the p.d. kernel as in (\ref{eq:n4}), and $\mathscr{H}_{K^{\left(\nu\right)}}$
the RKHS of $K$. Suppose $\left\{ Q_{n}\right\} _{\mathbb{N}_{0}}$
is an effective system in $\mathscr{H}_{K^{\left(\nu\right)}}$, and
let $\Psi$ be the isometry specified in (\ref{eq:7})--(\ref{eq:q8}).
Then 
\[
\left\{ \Psi Q_{n}\right\} _{n\in\mathbb{N}_{0}}
\]
is effective in the closed subspace 
\[
\Psi\left(\mathscr{H}_{K^{\left(\nu\right)}}\right)\subset L^{2}\left(\Omega,\mathbb{P}\right).
\]
\end{prop}

\begin{proof}
For all $F,G\in\mathscr{H}_{K^{\left(\nu\right)}}$, we have 
\begin{align*}
\mathbb{E}\left[\Psi\left(F\right)\Psi\left(G\right)\right] & =\left\langle F,G\right\rangle _{\mathscr{H}_{K^{\left(\nu\right)}}}\\
 & =\sum_{n\in\mathbb{N}_{0}}\left\langle Q_{n}F,Q_{n}G\right\rangle _{\mathscr{H}_{K^{\left(\nu\right)}}}\\
 & =\sum_{n\in\mathbb{N}_{0}}\mathbb{E}\left[\Psi\left(Q_{n}F\right)\Psi\left(Q_{n}G\right)\right].
\end{align*}
\end{proof}
\bibliographystyle{amsalpha}
\bibliography{ref}

\end{document}